\documentclass[10pt]{article}
\usepackage{amssymb,amsmath,amsopn}
\usepackage[dvips]{graphicx}   
\usepackage{color,epsfig}      
\usepackage{url}

\makeatletter
\def\@settitle{\begin{center}%
  \baselineskip14\p@\relax
  \bfseries
  \uppercasenonmath\@title
  \@title
  \ifx\@subtitle\@empty\else
     \\[1ex]\uppercasenonmath\@subtitle
     \footnotesize\mdseries\@subtitle
  \fi
  \end{center}%
}
\def\subtitle#1{\gdef\@subtitle{#1}}
\def\@subtitle{}
\makeatother

\newtheorem{thm}{Theorem}

\newtheorem{lem}{Lemma}

\newtheorem{defn}{Definition}
\newtheorem{prob}{Problem}

\newtheorem{rem}{Remark}

\DeclareMathOperator{\conv}{conv}

\DeclareMathOperator{\area}{area}

\DeclareMathOperator{\cl}{cl}

\DeclareMathOperator{\arcsinh}{arcsinh}

\newcommand{\HH}{\mathbb{H}}
\newcommand{\M}{\mathbb{M}}

\newcommand{\F}{\mathcal{F}}
\newcommand{\Eu}{\mathbb{E}}

\newcommand{\Sph}{\mathbb{S}}

\renewcommand{\S}{\mathcal{S}}

\title{On optimal $\lambda$-separable packings in the plane
\footnote{\hspace*{.35cm} 2010 Mathematics Subject Classification: 52A55, 52A40, 52C15.}}
\author{K\'{a}roly Bezdek\thanks{Partially supported by a Natural Sciences and 
Engineering Research Council of Canada Discovery Grant.} and Zsolt L\'angi\thanks{Partially supported by the National Research, Development and Innovation Office, NKFI, K-119670 and BME Water Sciences \& Disaster Prevention TKP2020 Institution Excellence Subprogram, grant no. TKP2020 BME-IKA-VIZ.}
}

\date{}

\begin{document}

\maketitle

\begin{abstract}
Let $\mathcal{P}$ be a packing of circular disks of radius $\rho>0$ in the Euclidean, spherical, or hyperbolic plane. Let $0\leq\lambda\leq\rho$. We say that $\mathcal{P}$ is a {\it $\lambda$-separable packing of circular disks of radius $\rho$} if the family $\mathcal{P'}$ of disks concentric to the disks of $\mathcal{P}$ having radius $\lambda$ form a totally separable packing, i.e.,  any two disks of $\mathcal{P'}$ can be separated by a line which is disjoint from the interior of every disk of $\mathcal{F'}$. This notion bridges packings of circular disks of radius $\rho$ (with $\lambda=0$) and totally separable packings of circular disks of radius $\rho$ (with $\lambda=\rho$). In this note we extend several theorems on the {\it density, tightness, and contact numbers} of disk packings and totally separable disk packings to $\lambda$-separable packings of circular disks of radius $\rho$ in the {\it Euclidean, spherical, and hyperbolic plane}. In particular, our upper bounds (resp., lower bounds) for the density (resp., tightness) of $\lambda$-separable packings of unit disks in the Euclidean plane are sharp for all $0\leq\lambda\leq 1$ with the extremal values achieved by $\lambda$-separable lattice packings of unit disks. On the other hand, the bounds of similar results in the spherical and hyperbolic planes are not sharp for all $0\leq\lambda\leq\rho$ although they do not seem to be far from the relevant optimal bounds either. The proofs use local analytic and elementary geometry and are based on the so-called {\it refined Moln\'ar decomposition}, which is obtained from the underlying Delaunay decomposition and as such might be of independent interest.
\end{abstract}

\section{Introduction}

Let $\M\in \{ \Eu^2, \HH^2, \Sph^2 \}$ denote the Euclidean, hyperbolic, or spherical plane, i.e., one of the simply connected complete $2$-dimensional Riemannian manifolds of constant sectional curvature. Since simply connected complete space forms, the sectional curvature of which have the same sign are similar, we may assume without loss of generality that the sectional curvature of $\M$ is $0, -1,$ or $1$.

Recall that finding the densest packing of congruent (circular) disks in $\M$ is a classical problem of discrete geometry that has been investigated in great details with the basic results published in L. Fejes T\'oth's classical book \cite{Fe}. The concept of totally separable packings is more recent. It was introduced by G. Fejes T\'oth and L. Fejes T\'oth in \cite{FeFe} and attracted significant attention. According to their definition a packing $\mathcal{P}$ of unit disks in the Euclidean plane is said to be totally separable if any two unit disks of $\mathcal{P}$ can be separated by a line which is disjoint from the interior of every unit disk of $\mathcal{F}$. This motivates the following definition that bridges packings and totally separable packings as follows.

\begin{defn}
Let $0\leq \lambda\leq \rho$. Let $\mathcal{P}$ be a packing of (circular) disks of radius $\rho$ in $\M\in \{ \Eu^2, \HH^2, \Sph^2 \}$. We say that $\mathcal{P}$ is a {\rm $\lambda$-separable packing of disks of radius $\rho$} if the family $\mathcal{P'}$ of disks concentric to the disks of $\mathcal{P}$ having radius $\lambda$ form a totally separable packing in $\M$, i.e.,  any two disks of $\mathcal{P'}$ can be separated by a line in $\M$ which is disjoint from the interior of every disk of $\mathcal{F'}$. 
\end{defn}

\begin{rem}
Clearly, a $0$-separable packing of disks of radius $\rho$ in $\M\in \{ \Eu^2, \HH^2, \Sph^2 \}$ is simply a packing of disks of radius $\rho$ in $\M$. On the other hand, any $\rho$-separable packing of disks of radius $\rho$ in $\M\in \{ \Eu^2, \HH^2, \Sph^2 \}$ is a totally separable packing of disks of radius $\rho$ in $\M$.
\end{rem} 

In this note we are going to investigate $\lambda$-separable packings of disks of radius $\rho$ in $\M\in \{ \Eu^2, \HH^2, \Sph^2 \}$ from the point of view of density, tightness, and contact numbers. Before stating our main theorems, we introduce some notation.

Let
\begin{equation}\label{eq:x1s}
x_1^s(y) := \frac{1}{2} \arcsin \frac{\cos \lambda \sin^2 y}{\sqrt{\sin^2 y - \sin^2 \lambda}}, \hbox{ if } 0 < \lambda < \frac{\pi}{4}, \arcsin \tan \lambda < y < \frac{\pi}{2},
\end{equation}
\begin{equation}\label{eq:x2s}
x_2^s(y) := \frac{\pi}{2}-\frac{1}{2} \arcsin \frac{\cos \lambda \sin^2 y}{\sqrt{\sin^2 y - \sin^2 \lambda}}, \hbox{ if } 0 < \lambda < \frac{\pi}{4}, \arcsin \tan \lambda < y < \frac{\pi}{2},
\end{equation}
\begin{equation}\label{eq:xh}
x^h(y) := \frac{1}{2} \arcsinh \frac{\cosh \lambda \sinh^2 y}{\sqrt{\sinh^2 y - \sinh^2 \lambda}}, \hbox{ if } 0 < \lambda < y,
\end{equation}
\begin{equation}\label{eq:xeu}
x^e(y) := \frac{y^2}{2\sqrt{y^2-\lambda^2}}, \hbox{ if } 0 < \lambda < y.
\end{equation}

Furthermore, for any $0 < \lambda < \frac{\pi}{4}, \arcsin \tan \lambda < y < \frac{\pi}{2}$ and $i=1,2$ let $T_i^s(y)$ denote the spherical isosceles triangle of edge lengths $2y, 2x_i^s(y), 2x_i^s(y)$, and let $R_i^s(y)$ denote the circumradius of $T_i^s(y)$. Similarly, for any $0 < \lambda < y$, let $T^h(y)$ denote the hyperbolic isosceles triangle with edge lengths $2y, 2x^h(y), 2x^h(y)$, and let $R^h(y)$ denote the circumradius of $T^h(y)$. We denote by $T^e(y)$ the Euclidean isosceles triangle with edge lengths $2y, 2x^e(y), 2x^e(y)$ and its circumradius by $R^e(y)$. In addition, we denote the regular spherical, hyperbolic, Euclidean triangle of edge length $2\rho$ by $T_{reg}^s(\rho)$,$T_{reg}^h(\rho)$ and $T_{reg}^e(\rho)$, respectively, and the circumradii of these triangles by $R_{reg}^s(\rho)$, $R_{reg}^h(\rho)$, and $R_{reg}^e(\rho)$, respectively.

\begin{rem}\label{rem:monotonicityofx}
It is an elementary computation to check that $x_1^s$ as a function of $y$ is strictly decreasing over the closed interval $S_1:= [ \arcsin \tan \lambda, \arcsin (\sqrt{2} \sin \lambda) ]$ and strictly increasing over $S_2:= [\arcsin (\sqrt{2} \sin \lambda), \pi/2]$, with $x_1^s(\arcsin \tan \lambda) = x_1^s (\pi/2) = \frac{\pi}{4}$ and $x_1^s(\arcsin (\sqrt{2} \sin \lambda)) = \lambda$. Since $x_2^s(y)=\frac{\pi}{2}-x_1^s(y)$, similar statements hold for $x_2^s$. For $k=1,2$ and $j=1,2$, we denote the inverse of the restriction of $x_j^s$ to $S_k$ by $x_j^s|_{S_k}^{-1}$. We remark that $\arcsin \tan \lambda < \arcsin (\sqrt{2} \sin \lambda) < \frac{\pi}{2}$ holds for all $\lambda \in (0,\pi/4)$. Furthermore, $x^h(y)$ is a strictly decreasing function of $y$ on $H_1:=(\lambda, \arcsinh (\sqrt{2} \sinh \lambda)]$, and strictly increasing on $H_2:=[\arcsinh (\sqrt{2} \sinh \lambda),+\infty)$, with $x^h(y) \to +\infty$ as $y \to \lambda^+$ or $y \to +\infty$. We denote the inverse of $x^h$ on the two intervals by $x^h|_{H_1}^{-1}$ and $x^h|_{H_2}^{-1}$, respectively.
\end{rem}

For any $0 < \lambda \leq \arcsin \frac{3}{5}$, let 
\begin{equation}\label{eq:yss}
y_s^s(\lambda):=\arcsin \sqrt{\frac{3+5\sin^2 \lambda - \sqrt{9-34 \sin^2 \lambda + 25 \sin^4 \lambda}}{8}} \hbox{ and}
\end{equation}
\begin{equation}\label{eq:ybs}
y_b^s(\lambda):=\arcsin \sqrt{\frac{3+5\sin^2 \lambda + \sqrt{9-34 \sin^2 \lambda + 25 \sin^4 \lambda}}{8}},
\end{equation}
and for any $\lambda > 0$, let
\begin{equation}\label{eq:ysh}
y_s^h(\lambda):=\arcsinh \sqrt{ \frac{5 \sinh^2 \lambda - 3 + \sqrt{25 \sinh^4 \lambda + 34 \sinh^2 \lambda + 9}}{8} }.
\end{equation}

Finally, for any $0 < \lambda < \frac{\pi}{2}$, we set
\begin{equation}\label{eq:crsols}
y_{min}^s(\lambda) := \arcsin\left( \frac{1}{6} \sqrt{ 6A^{1/3} - \frac{216 \left( -\frac{10}{9}L^4 +\frac{2}{3}L^2 \right) }{A^{1/3}} + 60L^2} \right),
\end{equation}
where $L := \sin \lambda$ and $A:=100 L^6-36 L^4+12 \sqrt{ -375 L^{12}+750 L^{10}-471 L^8+96 L^6}$.
Similarly, for any $0 < \lambda$, we set
\begin{equation}\label{eq:crsolh}
y_{min}^h(\lambda) := \arcsinh \sqrt{ \frac{5L^2}{3} + \frac{2}{3} \sqrt{10L^4+6L^2} \cos\left( \frac{1}{3} \arccos \left( \frac{(25 L^2+9) L}{4 \sqrt{2} (5L^2+3)^{3/2}} \right) -\frac{2\pi}{3} \right)  },
\end{equation}
where $L:=\sinh \lambda$.

An elementary computation shows that the above expressions exist on the required intervals. 


\begin{defn}
Let $B = \{ x \in \Eu^2 : ||x|| \leq 1 \}$ denote the unit disk centered at the origin of $\Eu^2$, where $||\cdot||$ stands for the Euclidean norm of the corresponding vector (resp., point) in $\Eu^2$. Let $\delta_{\lambda}(B)$ (resp., $\delta_{\lambda}^*(B)$) denote the largest density of $\lambda$-separable packings of unit disks (resp., $\lambda$-separable lattice packings of unit disks) in $\Eu^2$, i.e., the largest fraction of $\Eu^2$ covered by $\lambda$-separable packings of unit disks (resp., $\lambda$-separable lattice packings of unit disks) in $\Eu^2$.
\end{defn}
It was proved in \cite{FeFe} that $\delta_{1}(B) = \delta_{1}^*(B) =\frac{\pi}{4}$. On the other hand, it is well known \cite{Fe} that $\delta_{0}(B) = \delta_{0}^*(B) =\frac{\pi}{\sqrt{12}}$. The following theorem extends these results to $\lambda$-separable packings as follows.
\begin{thm}\label{thm:densityEu}
Let $B = \{ x \in \Eu^2 : ||x|| \leq 1 \}$. Then
\[
\delta_{\lambda}(B) = \delta_{\lambda}^*(B) =
\left\{ \begin{array}{l}
\frac{\pi}{\sqrt{12}}, \hbox{ if } 0 \leq \lambda \leq \frac{\sqrt{3}}{2},\\
\frac{\pi}{4 \lambda}, \hbox{ if } \frac{\sqrt{3}}{2} \leq \lambda \leq 1.\\
\end{array}
\right.
\]
Furthermore, for $0 \leq \lambda \leq \frac{\sqrt{3}}{2}$ and $\frac{\sqrt{3}}{2} \leq \lambda \leq 1$ the lattice packing $\mathcal{F}$ of unit disks whose Delaunay triangles  are $T_{reg}^e(1)$ and $T^e(\sqrt{2-2\sqrt{1-\lambda^2}})$, respectively, is a $\lambda$-separable packing with density $\delta_{\lambda}(B)$ (cf. Figure~\ref{fig:density_Eu}).
\end{thm}

\begin{figure}[ht]
  \begin{center}
  \includegraphics[width=0.5\textwidth]{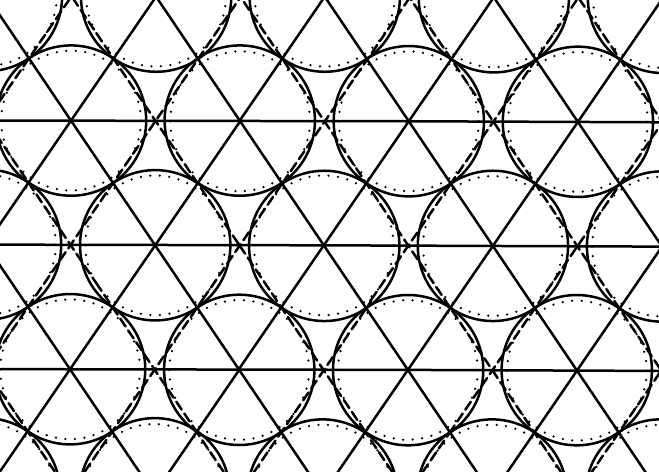} 
 \caption{A densest $\lambda$-separable packing of unit disks in $\Eu^2$ with $\lambda=0.93$. The unit disks and the sides of the Delaunay triangles are denoted by solid lines. The concentric disks of radius $\lambda$ and the lines separating them are drawn with dotted and dashed lines, respectively. The bases of the Delaunay triangles are horizontal, and their length is greater than $2$. The legs of the Delaunay triangles are of length $2$.}
\label{fig:density_Eu}
\end{center}
\end{figure}

\begin{defn}
Let $\gamma_{\lambda}(B)$ (resp., $\gamma_{\lambda}^*(B)$) denote the smallest $r>0$ such that there exists a $\lambda$-separable unit disk packing $B+X$ (resp., a $\lambda$-separable lattice packing $B+\Lambda$ of unit disks, where $\Lambda$ denotes a $2$-dimensional lattice in $\Eu^2$) satisfying $ \Eu^2 =rB+X$ (resp., $ \Eu^2 =rB+\Lambda$). We call $\gamma_{\lambda}(B)$ (resp., $\gamma_{\lambda}^*(B)$) the {\rm $\lambda$-separable tightness constant} (resp., {\rm $\lambda$-separable lattice tightness constant}) of $B$. In particular, if $\mathcal{P}:=B+X$ is a packing of unit disks in $\Eu^2$, then the smallest $r>0$ satisfying $ \Eu^2 =rB+X$ is called the {\rm tightness} of $\mathcal{P}$.
\end{defn}
The notion of tightness constant goes back to Rogers (see \cite{Zo}) and it is often called the simultaneous packing and covering constant. Closely related notions were introduced by Ryskov \cite{Ry} as well as L. Fejes T\'oth \cite{Fe78}. It was noted in \cite{Fe78} that $\gamma_{0}(B) = \gamma_{0}^*(B) =\frac{2}{\sqrt{3}}$. The following theorem is an extension of that to $\lambda$-separable packings.

\begin{thm}\label{thm:tightnessEu}
Let $B = \{ x \in \Eu^2 : ||x|| \leq 1 \}$. Then
\[
\gamma_{\lambda}(B) = \gamma_{\lambda}^*(B) =
\left\{ \begin{array}{l}
\frac{2}{\sqrt{3}}, \hbox{ if } 0 \leq \lambda \leq \frac{\sqrt{3}}{2},\\
\frac{\sqrt{2-2\sqrt{1-\lambda^2}}}{\lambda}, \hbox{ if } \frac{\sqrt{3}}{2} \leq \lambda \leq \frac{2\sqrt{2}}{3},\\
\frac{3 \sqrt{3} \lambda}{4}, \hbox{ if } \frac{2\sqrt{2}}{3} \leq \lambda \leq 1.\\
\end{array}
\right.
\]
Furthermore, the tightness of a $\lambda$-separable packing $\mathcal{F}$ of unit disks is $\gamma_{\lambda}(B)$ if and only if the Delaunay triangles defined by $\mathcal{F}$ tile $\Eu^2$ and are all congruent to $T_{reg}^e(1)$, $T^e \left(\sqrt{2-2\sqrt{1-\lambda^2}} \right)$ and $T^e\left( \sqrt{\frac{3}{2}} \lambda \right)$, respectively (see Figure~\ref{fig:tightness_Eu}).
\end{thm}

\begin{figure}[ht]
  \begin{center}
  \includegraphics[width=0.9\textwidth]{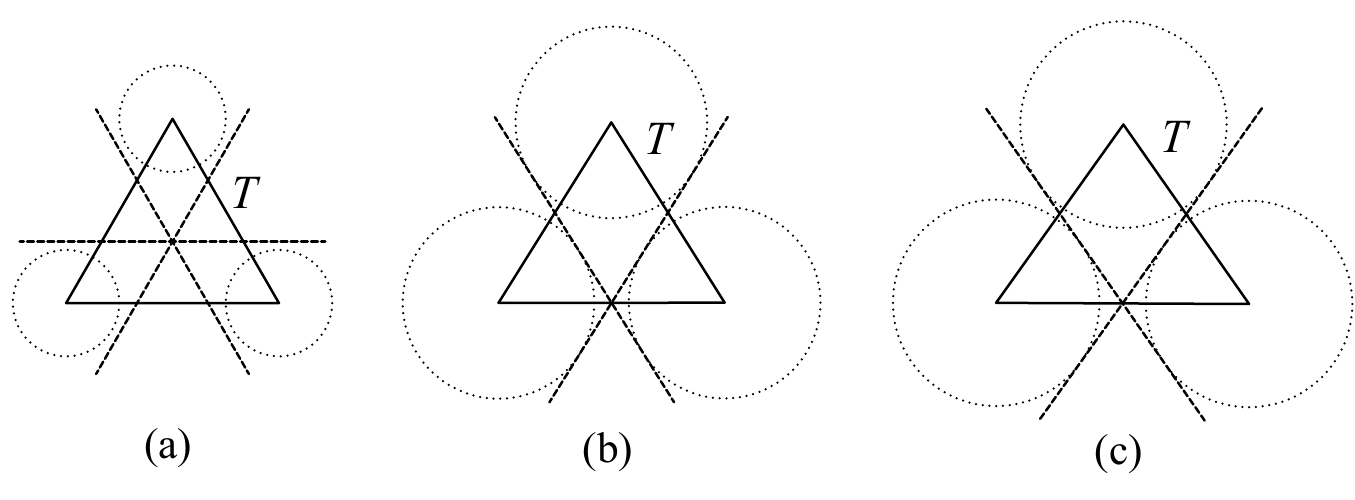} 
 \caption{The Delaunay triangle $T$ of three $\lambda$-separable unit disks in $\Eu^2$ for $\lambda =0.5<\frac{\sqrt{3}}{2}$ (case (a)), $\lambda=0,9 \in \left( \frac{\sqrt{3}}{2}, \frac{2\sqrt{2}}{3} \right)$ (case (b)) and $\lambda=0.97 > \frac{2\sqrt{2}}{3}$ (case (c)). The disks of radius $\lambda$ centered at the vertices of the triangles, and the lines separating them are drawn by dotted and dashed lines, respectively. The triangle $T$ in case (a) is a regular triangle of edge length $2$, the one in case (b) is an isosceles triangle of edge lengths $2$, $2$ and $2\sqrt{2-2\sqrt{1-\lambda^2}}$, and the one in case (c) is an isosceles triangle with edge lengths $\frac{3\lambda}{\sqrt{2}}$, $\frac{3 \lambda}{\sqrt{2}}$ and $\sqrt{6} \lambda$.}
\label{fig:tightness_Eu}
\end{center}
\end{figure}

According to \cite{Ha} the number of touching pairs in a packing of $n>1$ unit disks in $\Eu^2$ (called the {\it contact number} of the given unit disk packing) is at most $\lfloor 3n -\sqrt{12n-3}\rfloor$, where $\lfloor\cdot\rfloor$ denotes the lower integer part of the given real. On the other hand, it is proved in \cite{BeSzSz} (see also \cite{Be21}) that the contact number of an arbitrary totally separable packing of $n>1$ unit disks in $\Eu^2$ is at most $\lfloor 2n - 2\sqrt{n}\rfloor$. Both upper bounds are sharp and one can characterize the extremal configurations (see \cite{HeRa} and \cite{Be21}). For a comprehensive survey on the contact numbers, see \cite{BK}. The problem of finding a sharp upper bound for the contact numbers of $\lambda$-separable packings of $n>1$ unit disks in $\Eu^2$ seems to be a difficult question. We have only the following result, which is obtained using the proof technique of Theorem 4 of Eppstein \cite{Ep} combined with Theorem~\ref{thm:densityEu}.

\begin{defn}
Let $c_{\lambda}(n, B)$ denote the largest number of touching pairs of unit disks in a $\lambda$-separable packing of $n>1$ unit disks in $\Eu^2$.
\end{defn}

\begin{thm}\label{contact-upper-bound}\text{}
\begin{itemize}
\item[(i)] Let $0\leq \lambda\leq\frac{\sqrt{3}}{2}$. Then $c_{\lambda}(n, B)=\lfloor 3n -\sqrt{12n-3}\rfloor$ holds for all $n>1$.
\item[(ii)] Let $\frac{\sqrt{3}}{2}< \lambda\leq 1$. Then $\lfloor 2n - 2\sqrt{n}\rfloor\leq c_{\lambda}(n, B)\leq 2n -\sqrt{\pi\lambda}\sqrt{n}+O(1)$ holds for $n>1$, where $1.6494...=\sqrt{\pi\frac{\sqrt{3}}{2}}<\sqrt{\pi\lambda}\leq\sqrt{\pi}=1.7724...$.
\end{itemize}
\end{thm}

\begin{prob}\label{max-contact}
 Let $\frac{\sqrt{3}}{2}< \lambda\leq 1$. Then prove or disprove that $c_{\lambda}(n, B)=\lfloor 2n - 2\sqrt{n}\rfloor$ holds for all $n>1$. 
\end{prob}

\begin{rem}
Recall the following more general question of Swanepoel \cite{Sw}: Is it true that the number of touching pairs in any packing of $n>1$ unit disks in $\Eu^2$ having no touching triplets (i.e., whose contact graph is triangle-free) is at most $\lfloor 2n - 2\sqrt{n}\rfloor$? We note that if $\frac{\sqrt{3}}{2}< \lambda\leq 1$, then clearly in any $\lambda$-separable packing of $n>1$ unit disks in $\Eu^2$ there are no touching triplets. Thus, if the answer to Swanepoel's question is affirmative, then so is the answer to Problem~\ref{max-contact}. 
\end{rem}

In the next two theorems we deal with $\lambda$-separable packings of spherical caps of radius $\rho$ in $ \Sph^2 $. Clearly, if such a packing contains at least three caps, then $\lambda \leq \frac{\pi}{4}$, and we will see that in this case the inequality $\lambda \leq \frac{\pi}{2}-\rho$ is also satisfied. Totally separable packings of congruent spherical caps have been investigated in the recent paper \cite{BL23}. The next theorem extends Part (i) of Theorem 1.10 in \cite{BL23} to $\lambda$-separable packings of spherical caps of radius $\rho$  in $ \Sph^2 $ as follows.

\begin{defn}
Let $\delta_\lambda^s(\rho)$ denote the largest density of $\lambda$-separable packings of spherical caps of radius $\rho$ in $ \Sph^2 $, i.e., the largest fraction of $ \Sph^2 $ covered by a $\lambda$-separable packing of spherical caps of radius $\rho$ in $ \Sph^2 $.
\end{defn}

\begin{thm}\label{thm:densitySph}
Let $0 \leq \lambda \leq \rho < \frac{\pi}{2}$ with $\rho >0$ and $\lambda \leq \frac{\pi}{2} - \rho$. Then
\begin{equation}\label{eq:densitylemSph_thm}
\delta_\lambda^s(\rho) \leq
\left\{
\begin{array}{l}
\delta(T_1^s(x_1^s|_{S_1}^{-1}(\rho))), \hbox{ if } \lambda \leq \rho \leq \min \left\{ y_s^s(\lambda), \frac{\pi}{4} \right\}, \\
\delta(T_{reg}^s(\rho)), \hbox{ if } y_s^s(\lambda) < \rho \leq y_b^s(\lambda),\\
\delta(T_2^s(\rho)), \hbox{ if } \frac{\pi}{4} < \rho, \hbox{ and } \rho < y_s^s(\lambda) \hbox{ or } \rho > x_b^s(\lambda).
\end{array}
\right.                                                                      
\end{equation}
Furthermore, in the above three cases we have equality if and only if there is a $\lambda$-separable packing $\F$ of spherical caps of radius $\rho$ whose Delaunay triangles tile $\Sph^2$ and are congruent to $T_1^s(x_1^s|_{S_1}^{-1}(\rho))$, $T_{reg}^s(\rho)$, or $T_2^s(\rho)$, respectively.
\end{thm}

\begin{defn}
Let $\gamma_{\lambda}^s(\rho)$ denote the smallest $r>0$ such that there exists a $\lambda$-separable packing $\mathcal{P}$ of spherical caps of radius $\rho$ in $ \Sph^2 $ having the property that the spherical caps of radius $r$ concentric to the spherical caps of $\mathcal{P}$ cover $ \Sph^2 $.
\end{defn}

The following theorem extends earlier investigations on $\gamma_{0}^s(\rho)$ in \cite{Fe76} to $\lambda$-separable packings of spherical caps of radius $\rho$ in $ \Sph^2 $ as follows.

\begin{thm}\label{thm:tightnessSph}
Let $0 \leq \lambda \leq \rho < \frac{\pi}{2}$ with $\rho >0$ and $\lambda \leq \frac{\pi}{2} - \rho$. Then
\begin{equation}\label{eq:crlemSph_thm}
\gamma_{\lambda}^s(\rho) \geq
\left\{
\begin{array}{l}
R_1^s(y_{min}^s(\lambda)), \hbox{ if } \lambda \leq \rho \leq x_1^s(y_{min}^s(\lambda)),\\
R_1^s(x_1^s|_{S_1}^{-1}(\rho)), \hbox{ if } x_1^s(y_{min}^s(\lambda)) < \rho \leq \min  \left\{ y_s(\lambda), \frac{\pi}{4} \right\},\\
R_{reg}^s(\rho), \hbox{ if } y_s^s(\lambda) < \rho \leq y_b^s(\lambda),\\
R_2^s(\rho), \hbox{ if } \frac{\pi}{4} < \rho, \hbox{ and } \rho < y_s^s(\lambda) \hbox{ or } \rho > x_b^s(\lambda).
\end{array}
\right.                                                                      
\end{equation} 
Furthermore, in the above four cases we have equality if and only if there is a $\lambda$-separable packing $\F$ of spherical caps of radius $\rho$ whose Delaunay triangles tile $\Sph^2$ and are congruent to $T_1^s(y_{min}^s(\lambda))$, $T_1^s(x_1^s|_{S_1}^{-1}(\rho))$, $T_{reg}^s(\rho)$, or $T_2^s(\rho)$, respectively.
\end{thm}

\begin{rem}

Using the characterization of tilings of $\Sph^2$ with congruent spherical (isosceles) triangles (see \cite{D} and \cite{UA}), one may find pairs $(\lambda,\rho)$ with $0 < \lambda < \rho$ where our estimates are sharp. In particular, an elementary computation shows that the tilings of $\Sph^2$ generated by the isosceles triangles $H_{16}$ and $H_{20}$ in \cite{UA} are the Delaunay decompositions of $\lambda$-separable packings of spherical caps of radius $\rho$ with both maximal density and minimal tightness. The values of $\lambda$ and $\rho$ are
\[
\rho= \frac{1}{2} \arcsin\left(  \sqrt{2 \sqrt{2}-2} \right) \approx  0.57186,  \quad  \lambda= \arcsin \left(   \frac{ 2 \sin
\frac{\pi}{8}  \sqrt{1 + \sqrt{2} }    }{  \sqrt{4 + \sqrt{2}} } \right) \approx 0.53644 \quad \hbox{for } H_{16},
\]
\[
\rho = \frac{1}{2} \arcsin \left(  \frac{\sqrt{ 1 + 2\cos \frac{\pi}{5}}}{  1 + \cos \frac{\pi}{5} } \right) \approx
0.55357, \quad \lambda = \arcsin \left(  \frac{2}{\sqrt{5}}    \sin \frac{\pi}{10}
\sqrt{ 1 + 2 \cos \frac{\pi}{5} }    \right) \approx 0.46365 \quad \hbox{for } H_{20}.
\]
In addition, for $n=4,6,12$,  let $\F_n$ denote the family of $n$ spherical caps of radius $\rho_n$ with $\rho_4=\arcsin \sqrt{\frac{2}{3}}$, $\rho_6=\frac{\pi}{4}$, $\rho_{12}=\arcsin \frac{1}{2 \sin \frac{2\pi}{5}}$, where the centers of the caps are the vertices of a regular tetrahedron, octahedron or icosahedron inscribed in $\Sph^2$, respectively. The Delaunay triangles of $\F_n$ are regular triangles of edge length $2\rho_n$, and thus, they are $\lambda$-separable families of maximal density and minimal tightness for any value of $\lambda$ satisfying $y_s^s(\lambda) < \rho \leq y_b^s(\lambda)$.
\end{rem}

Finally, in the next two theorems we deal with $\lambda$-separable packings of hyperbolic disks of radius $\rho$ in $\HH^2$. 

\begin{defn}
Let $\delta_\lambda^h(\rho)$ denote the largest density of $\lambda$-separable packings of hyperbolic disks of radius $\rho$ in the cells of the hyperbolic refined Moln\'ar decompositions in $\HH^2$. 
\end{defn}
 
It has been observed by B\"or\"oczky \cite{Bo78} that the notion of density of sphere packings in hyperbolic space has to to be introduced with respect to well-defined underlying cell decompositions of the hyperbolic space (such as the hyperbolic refined Moln\'ar decomposition introduced in the Appendix).

\begin{thm}\label{thm:densityH}
Let $0 \leq \lambda \leq \rho$, where $0 < \rho$. Then
\begin{equation}\label{eq:densitylemH_thm}
\delta_\lambda^h(\rho) \leq
\left\{
\begin{array}{l}
\delta(T^h(x^h|_{H_1}^{-1}(\rho))), \hbox{ if } \lambda \leq \rho \leq y_s^h(\lambda), \\
\delta(T_{reg}^s(\rho)), \hbox{ if } y_s^h(\lambda) < \rho.\\
\end{array}
\right.                                                                     
\end{equation}
Furthermore, in the above two cases we have equality if there is a $\lambda$-separable packing of hyperbolic disks of radius $\rho$ whose Delaunay triangles are congruent to $T^h(x^h|_{H_1}^{-1}(\rho))$ or $T_{reg}^h(\rho)$, respectively.
\end{thm}

\begin{defn}
Let $\gamma_{\lambda}^h(\rho)$ denote the smallest $r>0$ such that there exists a $\lambda$-separable packing $\mathcal{P}$ of hyperbolic disks of radius $\rho$ in $\HH^2$ having the property that the hyperbolic disks of radius $r$ concentric to the hyperbolic disks of $\mathcal{P}$ cover $\HH^2$.
\end{defn}

\begin{thm}\label{thm:tightnessH}
Let $0 \leq \lambda \leq \rho$, where $0 < \rho$. Then
\begin{equation}\label{eq:crlemH_thm}
\gamma_\lambda^s(\rho) \geq
\left\{
\begin{array}{l}
R^h(y_{min}^h(\lambda)), \hbox{ if } \lambda \leq \rho \leq x^h(y_{min}^h(\lambda)),\\
R^h(x^h|_{H_1}^{-1}(\rho)), \hbox{ if } x^h(y_{min}^h(\lambda)) < \rho \leq y_s^h(\lambda),\\
R_{reg}^h(\rho), \hbox{ if } y_s^h(\lambda) < \rho.
\end{array}
\right.                                                                      
\end{equation}
Furthermore, in the above three cases we have equality if there is a $\lambda$-separable packing of hyperbolic disks of radius $\rho$ whose Delaunay triangles tile $\HH^2$ and are congruent to $T^h(y_{min}^h(\lambda))$, $T^h(x^h|_{H_1}^{-1}(\rho))$, or $T_{reg}^h(\rho)$, respectively.
\end{thm}

In the rest of the paper we prove the theorems stated. The proofs use local analytic and elementary geometry and are based on the so-called refined Moln\'ar decompositions.  We note that the Euclidean Moln\'ar decomposition was introduced by Moln\'ar in \cite{Molnar}, the spherical refined analogue of which was introduced and applied in \cite{BL23}. For the sake of completeness, the Appendix of this paper gives a description of the hyperbolic refined Moln\'ar decomposition, the method of which extends to the Euclidean as well as the spherical plane.

In the proofs, for any two points $p,q$ in $\M $ with $\M \in \{ \Eu^2, \HH^2, \Sph^2 \}$, which are assumed not to be antipodal if $\M = \Sph^2$, we denote the unique shortest geodesic segment connecting them by $[p,q]$. The distance of $p$ from $q$ is labelled by $d(p,q)$, and $\conv(X)$ (resp., $\area(X)$) stands for the convex hull (resp., the area) of the set $X \subset \M$. If $\M=\Sph^2$, $\conv (X)$ denotes the intersection of all closed hemispheres of $\Sph^2$ containing $X$, or if no such hemisphere exists, then we set $\conv(X) = \Sph^2$.

\section{Proofs of Theorems 4 and 5}\label{sec:proof45}


In the proof we use the refined Moln\'ar decomposition defined by a finite point set $X$ of $\Sph^2$, as in the proof of Theorem 2 of \cite{BL23}. We note that even though in the definition of this decomposition we assumed that no closed hemisphere contains $X$, the same argument can be applied to define it if there is a closed hemisphere of $\Sph^2$ containing $X$. In this case the union of the cells of the refined Moln\'ar decomposition is the spherical convex hull $\conv (X)$ of $X$.

We prove our statement in a more general setting. Namely, we say that a packing $\F$ of spherical caps of radius $\rho$ is \emph{$R$-locally $\lambda$-separable} if 
any triple of caps in $\F$ whose centers are at pairwise spherical distances at most $R$ are $\lambda$-separable.
In the proof we assume that $\F$ is a $(2 R_{\rho})$-locally $\lambda$-separable packing of spherical caps of radius $\rho$, where $R_{\rho}= \arcsin (\sqrt{2} \sin \rho)$ (if $\rho > \frac{\pi}{4}$, then $R_{\rho}$ does not exist, in this case we simply assume that $\F$ is $\lambda$-separable). In addition, we assume that $\F$ is \emph{$(2R_{\rho})$-saturated} in  $\Sph^2$, i.e., every point $p$ of $\Sph^2$ is at distance at most $2R_{\rho}$ from the center of a cap in $\F$, as otherwise the spherical cap of center $p$ and radius $\rho$ can be added to $\F_m$ preserving its $(2R_{\rho})$-separability. This implies that the circumradius of any cell in the Delaunay decomposition of the set $X$ of the centers of the elements of $\F$ is at most $2 R_{\rho}$.

Here we note that $R_{\rho}$ is the circumradius of a regular spherical quadrangle with edge length $2\rho$.

We decompose the cells of the $M$-decomposition of $\conv(X)$ into two types of cells $P$ as follows:
\begin{itemize}
\item[(i)] $P$ has circumradius at most $R_{\rho}$, in this case we say that $P$ is type 1, or
\item[(ii)] it is of the form $P=\cl \left( \conv \{ v,c_i,c_j\} \setminus \conv \{ v',c_i,c_j\} \right)$, where $c_i,c_j \in C$, $v$ is the circumcenter of a Delaunay cell with $c_i$ and $c_j$ as vertices and with circumradius at least $R_{\rho}$, and $d(v',c_i) = d(v',c_j)$. In this case we say that $P$ is type 2.
\end{itemize}

The above defined cell decomposition is called the refined $M$-decomposition of the packing.

We observe that since the sides of $P$ are of length at least $2\rho$, any type 1 spherical polygon $P$ is either a triangle containing its circumcenter, or a regular spherical quadrangle with edge length $2\rho$; the latter quadrangle can be decomposed into two triangles containing their circumcenter.

Let $P$ be a cell of the refined $M$-decomposition. If the sum of the internal angles of $P$ is denoted by $\varphi$, then we define the \emph{density of $\F$ in $P$} as
\[
\delta(P)= \frac{(1-\cos \rho) \varphi}{\area (P)}.
\] 
Clearly, to prove Theorem 4 (respectively, Theorem 5), it is sufficient to prove that for the density of $\F$ in any cell (respectively, circumradius of any cell) of the decomposition is less than or equal to the value promised in the theorem.

If $P$ is a type 2 cell of the decomposition, then this statement readily follows from \cite[Lemma 3]{BL23}, which we quote for completeness, without proof.

\begin{lem}\label{lem:BL23}
Let $Q$ be an isosceles spherical triangle, with vertices $q_1, q_2, p$, where $p$ is the apex. Let the length of the legs of $Q$ be $x$, and that of the base be $y$, where $2\rho\leq y < \pi$ and $\frac{x}{2}< y < \frac{\pi}{2}$. For $i=1,2$, let $S_i$ denote the spherical cap of radius $\rho$, centered at $p_i$. Let $f(x,y)$
denote the density of the packing $\{ S_1, S_2 \}$ in $Q$ as a function of $x$ and $y$. Then $f(x,y)$ is a strictly decreasing function of both $x$ and $y$.
\end{lem}

Finally, if $P$ is a type 1 cell of the decomposition, then the assertion follows from Lemmas~\ref{lem:densitySph} and \ref{lem:tightnessSph}.

\begin{lem}\label{lem:densitySph}
Let $T = \conv \{a,b,c\} \subset \Sph^2$ be a spherical triangle with edge lengths at least $2\rho$ that contains its circumcenter. For $* \in \{ a,b,c\}$, let $C_*$ denote the closed spherical cap of radius $\lambda \leq \rho$ centered at $*$. Assume that there are lines $L_a$ and $L_b$ such that $L_a$ separates $C_a$ from $C_b$ and $C_c$, and $L_b$ separates $C_b$ from $C_a$ and $C_c$. Let $\delta(T)$ denote the density of $\{ C_x, C_y, C_z \}$ with respect to $T$.
Then $0 \leq \lambda < \frac{\pi}{4}$, $\lambda \leq \rho \leq \frac{\pi}{2}-\rho$, and we have the following.
\begin{equation}\label{eq:densitylemSph}
\delta(T) \leq
\left\{
\begin{array}{l}
\delta(T_1^s(x_1^s|_{S_1}^{-1}(\rho))), \hbox{ if } \lambda \leq \rho \leq \min \left\{ y_s^s(\lambda), \frac{\pi}{4} \right\}, \\
\delta(T_{reg}^s(\rho)), \hbox{ if } y_s^s(\lambda) < \rho \leq y_b^s(\lambda),\\
\delta(T_2^s(\rho)), \hbox{ if } \frac{\pi}{4} < \rho, \hbox{ and } \rho < y_s^s(\lambda) \hbox{ or } \rho > x_b^s(\lambda).
\end{array}
\right.                                                                      
\end{equation}
Furthermore, in the above three cases we have equality if and only if $T$ is congruent to $T_1^s(x_1^s|_{S_1}^{-1}(\rho))$, $T_{reg}^s(\rho)$, or $T_2^s(\rho)$, respectively.
\end{lem}

\begin{lem}\label{lem:tightnessSph}
Let $T = \conv \{a,b,c\} \subset \Sph^2$ be a spherical triangle with edge lengths at least $2\rho$ that contains its circumcenter. For $* \in \{ a,b,c\}$, let $C_*$ denote the closed spherical cap of radius $\lambda \leq \rho$ centered at $*$. Assume that there are lines $L_a$ and $L_b$ such that $L_a$ separates $C_a$ from $C_b$ and $C_c$, and $L_b$ separates $C_b$ from $C_a$ and $C_c$. Let $R(T)$ denote the circumradius of $T$.
Then $0 \leq \lambda < \frac{\pi}{4}$, $\lambda \leq \rho \leq \frac{\pi}{2}-\lambda$, and we have the following.
\begin{equation}\label{eq:crlemSph}
R(T) \geq
\left\{
\begin{array}{l}
R_1^s(y_{min}^s(\lambda)), \hbox{ if } \lambda \leq \rho \leq x_1^s(y_{min}^s(\lambda)),\\
R_1^s(x_1^s|_{S_1}^{-1}(\rho)), \hbox{ if } x_1^s(y_{min}^s(\lambda)) < \rho \leq \min  \left\{ y_s(\lambda), \frac{\pi}{4} \right\},\\
R_{reg}^s(\rho), \hbox{ if } y_s^s(\lambda) < \rho \leq y_b^s(\lambda),\\
R_2^s(\rho), \hbox{ if } \frac{\pi}{4} < \rho, \hbox{ and } \rho < y_s^s(\lambda) \hbox{ or } \rho > x_b^s(\lambda).
\end{array}
\right.                                                                      
\end{equation}
Furthermore, in the above four cases we have equality if and only if $T$ is congruent to $T_1^s(y_{min}^s)$, $T_1^s(x_1^s|_{S_1}^{-1}(\rho))$, $T_{reg}^s(\rho)$, or $T_2^s(\rho)$, respectively.
\end{lem}

To prove these lemmas, we prove a preliminary lemma which is valid in any plane of constant curvature.

\begin{lem}\label{lem:trivial}
Let $T$ be a triangle in $\M \in \{ \Eu^2, \HH^2, \Sph^2 \}$ with vertices $a,b,c$ and containing its circumcenter. Let $C_a, C_b, C_c$ be disks of radius $\rho > 0$ centered at $a,b,c$, respectively. Assume that a line $L$ separates $C_a$ and $C_b$ from $C_c$. Then the line passing through the midpoints of $[a,c]$ and $[b,c]$ separates $C_a$ from $C_b$ and $C_c$.
\end{lem}

\begin{proof}
Let $\bar{L}$ be the line passing through the midpoints of $[a,c]$ and $[b,c]$.
Clearly, it is sufficient to prove Lemma~\ref{lem:trivial} under the additional assumption that there is no $\rho' > \rho$ such that some line $L'$ separates  $C_a'$ and $C_b'$ from $C_c'$, where $C_a'$, $C_b'$ and $C_c'$ are disks of radius $\rho' > \rho$ centered at $a,b,c$, respectively.
Thus, we may assume that $L$ cannot be perturbed in such a way that it is disjoint from all of $C_a, C_b, C_c$, implying that we have one of the following.
\begin{enumerate}
\item[(a*)] All of $C_a, C_b, C_c$ touches $L$.
\item[(b*)] Two of $C_a, C_b, C_c$ touches $L$ at the same point.
\item[(c*)] $\M= \Sph^2$, and two of $C_a, C_b, C_c$ touches $L$ at two antipodal points.
\end{enumerate}
In case (a*), $L$ passes through the midpoints of $[a,c]$ and $[b,c]$, i.e $\bar{L}=L$.
In case (b*), the two disks touch $L$ from opposite sides; that is, one of them is $C_c$, and we may assume that the other one is $C_a$. Then the distance of $b$ from $L$ is strictly larger than $\rho$, implying that $L$ strictly separates the midpoint of $[b,c]$ and $b$.
Thus, the midpoint of $[c,b]$ is closer to $a$ than to $c$ and $b$, from which we readily have that $T$ does not contain its circumcenter; a contradiction.

Finally, in case (c*) we may assume that the two disks touching $L$ are $C_a$ and $C_b$. Then, by the symmetry of the configuration one can modify $L$ in such a way that the perturbed line touches $C_a$ and $C_b$ at points which are not antipodal, and thus, there is a line disjoint from all of $C_a, C_b, C_c$ that separates $C_a, C_b$ from $C_c$; a contradiction.
\end{proof}

\begin{proof}[Proof of Lemmas~\ref{lem:densitySph} and \ref{lem:tightnessSph}]
Without loss of generality, we may assume that $\lambda > 0$.

First, by Girard's theorem relating the spherical excess of a triangle to its area, $\delta(T)$ is maximal if and only if $\area(T)$ is minimal among the triangles satisfying the conditions in Lemma~\ref{lem:densitySph}.

Note that since any side of $T$ is of length at least $2\rho$, and any vertex of $T$ is at distance at least $\lambda$ from a line passing through the midpoints of two sides, by compactness, both area and circumradius are minimized for some triangles in the family of triangles satisfying the conditions in the lemmas. Without loss of generality, we may assume that $T_{\delta}$ has maximal density (minimal area) and $T_{\gamma}$ has minimal circumradius in this family. For the moment, we deal only with Lemma~\ref{lem:densitySph}, and assume that $T=T_{\delta}$.

Let $G$ denote the Lexell circle of $T$ generated by $c$, that is, let $G$ be the circle containing $-a,-b,c$. Let $H$ denote the open hemisphere that contains $c$ in its interior, and $a,b$ in its boundary. Let $G_0 = H \cap G$. It is well known that if $d \in H$, then $\area(\conv \{ a,b,d \}) = \area(\conv \{ a,b,c \})$ if and only if $d \in G_0$. Let $L$ denote the bisector of the segment $[a,b]$, and let $c_0$ denote the intersection point of $C_0$ and $L$. Without loss of generality, we may assume that $c$ is not farther from $a$ than from $b$. Now we consider the possible position of all points $d \in C_0$ with the property that $T'=\conv \{a,b,d \}$ contains its circumcenter $o'$, its edge lengths are at least $2\rho$, and there is a line that separates $C_a$ from $C_b$ and $C_c$, and a line that separates $C_b$ from $C_a$ and $C_c$.

First, the property that $T'$ contains $o'$ is satisfied if and only if for any $\{ x,y,z \} = \{ a,b,d \}$, the distance of the midpoint of $[x,y]$ from $z$ is not less than from $x$ or $y$. This condition holds if and only if $d \in G_1$ for some closed arc $G_1$ in $G_0$ symmetric to $L$. Similarly, there is a closed arc $G_2 \subset G_0$ symmetric to $L$ such that $T'$ has edge lengths at least $2\rho$ if and only if $d \in G_2$.

Finally, let $L'_a$ denote the line through the midpoints of $[a,b]$ and $[a,c]$, and note that it separates $C_a$ from $C_b$ and $C_c$. Let $L'_b$ denote the reflected copy of $L'_a$ to $L$, and note that it separates $C_b$ from $C_a$ and $C_c$. Furthermore, for any point $d$ lying on the arc of $C$ connecting $c$ and its reflected copy to $L$, the disk of radius $\lambda$ and center $d$ is separated from $C_a$ by $L'_a$, and from $C_b$ by $L'_b$. More generally, there is a closed arc $G_3$ in $G_0$, symmetric to $L$, such that $C_a$ and $C_d$ is separated from $C_b$ by a line and $C_b$ and $C_d$ are separated from $C_a$ by another line if and only if $d \in G_3$. We remark that the fact that $T$ satisfies the above three properties implies that $G_1 \cap G_2 \cap G_3$ is not empty. Since this set is a closed arc in $G_0$ symmetric to $L$, it follows that $\conv \{ a,b,c_0 \}$ contains its circumcenter, has edge lengths at least $2\rho$, and it satisfies the separability properties described in the lemma. Thus, $T$ has minimal area only if $T$ is symmetric to $L$ and it satisfies at least one of the following.
\begin{enumerate}
\item[(a**)] Its circumcenter $\bar{o}$ lies on $[a,b]$.
\item[(b**)] $d(a,c) = d(b,c) = 2\rho$.
\item[(c**)] The lines $L_a'$ and $L_b'$ touch $C_a$, $C_b$ and $C_c$.
\end{enumerate}

Consider the case that $T$ satisfies (a**) but it does not satisfy (b**) or (c**). Then slightly moving all of $a,b,c$ towards $\bar{o}$ by the same quantity yields a triangle that has strictly smaller area and still satisfies the conditions in the lemma; a contradiction. Thus, we have that $T$ satisfies (b**) or (c**). If $T$ satisfies (b**) but it does not satisfy (c**) and the distance of $a$ and $b$ is greater than $2\rho$, then we may decrease the angle at $c$ while keeping the distance of $c$ from $a$ and $b$ fixed to obtain a similar contradiction. Hence, $T$ satisfies (c**), or $T$ is a regular triangle of edge length $2\rho$.

We observe that the same consideration can be carried out under the assumption that $T=T_{\gamma}$, in which the circumcircle of $T$ plays the role of the Lexell circle. Thus, we have that any triangle $T$ with minimal circumradius and satisfying the conditions in Lemma~\ref{lem:tightnessSph} either satisfies (c**), or it is a regular triangle of edge length $2\rho$.
It is also worth noting that any isosceles triangle satisfying the required conditions has sides of lengths at least $2\rho$. Thus, it is easy to show that if $T_{reg}^s(\rho)$ satisfies the conditions of the lemmas, then this triangle has minimal area and minimal circumradius.

From now on, we assume that $T$ is a triangle satisfying (c**) and the conditions of the lemmas, and investigate the properties of $T$.

Let the angles of $T$ be $2\alpha, \beta, \beta$ and its edge lengths be $2x, 2x, 2y$. It can be shown that in that case the triangle with angles $2\alpha, \pi-\beta, \pi-\beta$ and edge lengths $2y, \pi-2x, \pi-2x$ also satisfies (c**).

\emph{Step 1}: computing the edge lengths of the triangles satisfying (c**).

First, we note that $\lambda \leq \frac{\pi}{4}$ holds since there are two separating great circles that divide $\Sph^2$ into four lunes.

Let $T$ be an isosceles spherical triangle with edges of lengths $2x,2x,2y$ ($0 < x < y < \frac{\pi}{2}$), apex $p$, and vertices $q_1,q_2$ on its base. Assume that the line $L$ through the midpoints $m, m'$ of $[q_,q_2]$ and $[p,q_1]$, respectively, touches and separates the spherical caps of radius $\lambda$ centered at $p,q_1$.
Let the orthogonal projections of $p$ and $q_1$ on $L$ be $\bar{p}$ and $\bar{q}_1$, respectively. Let $w$ ,$z$ and $t$ be the lengths of the arcs $[m,m']$, $[\bar{p},m']$, and $[m,p]$, respectively. By the spherical Pythagorean Theorem, we have
\begin{eqnarray*}
\cos (2x) & = & \cos y \cos t,\\
\cos x & = & \cos \lambda \cos z,\\ 
\cos y & = & \cos \lambda \cos (w-z),\\
\cos t & = & \cos \lambda \cos (w+z),
\end{eqnarray*}
 From this we obtain that
\[
1 = \cos^2 w + \sin^2 w =  \frac{\left( \cos y + \cos t \right)^2}{4 \cos^2 \lambda \cos^2 z} + \frac{\left( \cos y - \cos t \right)^2}{4 \cos^2 \lambda \sin^2 z} = 
\frac{\left( \cos y + \cos t \right)^2}{4 \cos^2 x} + \frac{\left( \cos y - \cos t \right)^2}{4 (\cos^2 \lambda - \cos^2 x)},
\]
which yields that
\[
4 \cos^2 x (\cos^2 \lambda - \cos^2 x) \cos^2 y  = \left( \cos^2 y + \cos t \cos y \right)^2 (\cos^2 \lambda - \cos^2 x)  +
\left( \cos^2 y - \cos t \cos y \right)^2 \cos^2 x =
\]
\[
 \left( \cos^2 y + \cos (2x) \right)^2 (\cos^2 \lambda - \cos^2 x)  +
\left( \cos^2 y - \cos (2x) \right)^2 \cos^2 x =
\]
\[
 \left( \cos^2 y + \cos (2x) \right)^2 \cos^2 \lambda - 4 \cos^2 y \cos (2x) \cos^2 x.
\]
Since $\cos^2 \lambda - \cos^2 x + \cos (2x) = \cos^2 \lambda - \sin^2 x$, from this it follows that
\[
4 \cos^2 x (\cos^2 \lambda - \sin^2 x) \cos^2 y = \cos^4 y \cos^2 \lambda + 2 \cos^2 y \cos (2x) \cos^2 \lambda + \cos^2 (2x) \cos^2 \lambda.
\]
Thus,
\[
- 4 \cos^2 x \sin^2 x \cos^2 y = \cos^4 y \cos^2 \lambda - 2 \cos^2 y \cos^2 \lambda + \cos^2 (2x) \cos^2 \lambda, \hbox{ implying}
\]
\[
\sin^2 (2x) (\cos^2 \lambda - \cos^2 y) = \sin^4 y \cos^2 \lambda .
\]
As $0 < \lambda \leq \frac{\pi}{4}$, this yields that
\begin{equation}\label{eq:optimal}
\sin (2x) = \frac{\cos \lambda \sin^2 y}{\sqrt{\cos^2 \lambda - \cos^2 y}}.
\end{equation}
Thus, for any $0 < \lambda \leq \frac{\pi}{4}$, there is a triangle $T$ satisfying (c) if and only if $\arcsin \tan \lambda \leq y  < \frac{\pi}{2}$, and in this case the solutions, up to congruence, are $T_1^s(y)$ and $T_2^s(y)$.

\emph{Step 2}: Finding the values of $y$ such that $T_1^s(y)$ or $T_2^s(y)$ contains its circumcenter.

Note that if the height of an isosceles triangle $T$ corresponding to its base is at least $\frac{\pi}{2}$, then $T$ contains its circumcenter; in other words, $T_2^s(y)$ contains its circumcenter for all values of $y$. On the other hand, the triangle $T_1^s(y)$ contains its circumcenter if and only if the height corresponding to its base is of length at least $y$. Let $m$ denote the length of this height. Then, by the spherical Pythagorean theorem, we have $\cos m = \frac{\cos 2x_1^s(y)}{\cos y}$, implying that $T_1(y)$ contains its circumcenter if and only if $\cos 2x_1^s(y) \leq \cos ^2 y$. Since $x_1^s(y) \leq \frac{\pi}{4}$ and by (\ref{eq:optimal}), the latter inequality is equivalent to the inequality
\[
\sqrt{1- \frac{\cos^2 \lambda \sin^4 y}{\sqrt{\sin^2 y - \sin^2 \lambda}}} \leq \left( 1- \sin^2 y \right)^2.
\]
Since $0 < \lambda < \frac{\pi}{4}$, an elementary computation shows that this inequality is satisfied if and only if $0 \leq \sin^2 y \leq 2 \sin^2 \lambda$, implying that $\arcsin \tan \lambda \leq y \leq \arcsin (\sqrt{2} \sin \lambda)$. We note that $\arcsin (\sqrt{2} \sin \lambda) \geq \arcsin \tan \lambda$ is satisfied for all $0 \leq \lambda \leq \frac{\pi}{4}$.

\emph{Step 3}: Finding the relations between the lengths of the sides of $T_1^s(y)$, $T_2^s(y)$.

First, we consider $x_1^s(y)$ and investigate the inequality $x_1^s(y) \geq y$. Then we have $y \leq \frac{\pi}{4}$, and we need to solve the inequality
\[
\sin (2y) \leq \sin (2x_1^s)  = \frac{\cos \lambda \sin^2 y}{\sqrt{\sin^2 y - \sin^2 \lambda}}
\]
under the condition that $0 \leq x_1^s,y \leq \frac{\pi}{4}$. In this case both sides are nonnegative, and we can write the inequality as
\[
4 \sin^2 y (1-\sin^2 y) \leq \frac{(1-\sin^2 \lambda) \sin^4 y}{\sin^2 y - \sin^2 \lambda}.
\]
Multiplying both sides with the denominator and simplifying, we obtain that
\[
4(1-\sin^2 y)(\sin^2 y - \sin^2 \lambda) \leq (1-\sin^2 \lambda) \sin^2 y,
\]
and by algebraic transformations we obtain the inequality
\[
4 \sin^4 y -(3+5 \sin^2 \lambda) \sin^2 y + 4 \sin^2 \lambda \geq 0.
\]
For $\sin^2 y$ this inequality is a quadratic inequality, the discriminant of which is $9-34 \sin^2 \lambda + 25\sin^4 \lambda$. On the interval $\lambda \in [0,\pi/4]$, this expression is nonnegative if and only if $0 \leq \lambda \leq \arcsin (3/5)$. On the other hand, since by our conditions $\frac{\pi}{4} \geq y \geq \arcsin \tan \lambda$, we have $\sin \lambda \leq \frac{1}{\sqrt{3}} < \frac{3}{5}$. Thus, $x_1^s(y) \geq y$ implies that $0 \leq \lambda \leq \arcsin \frac{1}{\sqrt{3}}$, and the discriminant of the above inequality is positive. In this case the two roots are
\[
y_s^s(\lambda) = \arcsin \sqrt{\frac{3+5 \sin^2 \lambda - \sqrt{9-34 \sin^2 \lambda + 25\sin^4 \lambda}}{8}},
\]
\[
y_b^s(\lambda) = \arcsin \sqrt{\frac{3+5 \sin^2 \lambda + \sqrt{9-34 \sin^2 \lambda + 25\sin^4 \lambda}}{8}},
\]
where the above expressions can be shown to exist. Furthermore, numeric computations show that we have $\arcsin \tan \lambda \leq y_s^s(\lambda) \leq y_b^s(\lambda) \leq \frac{\pi}{2}$, $y_b^s(\lambda) \geq \frac{\pi}{4}$ for all values of $\lambda$, and $y_s^s(\lambda) \leq \frac{\pi}{4}$ if and only if $0 \leq \lambda \leq \arcsin \frac{1}{\sqrt{3}}$.

The investigation of the inequality $x_2^s(y) \leq y$ can be done in a similar way, and we obtain the following.

\begin{itemize}
\item If $0 \leq \lambda \leq \arcsin \frac{1}{\sqrt{3}}$ and $\arcsin \tan \lambda \leq y \leq y_s$, then $y \leq x_1(y) \leq x_2(y)$.
\item If $0 \leq \lambda \leq \arcsin \frac{3}{5}$ and $y_s \leq y \leq y_b$, then $x_1(y) \leq y \leq x_2(y)$.
\item Otherwise $x_1(y) \leq x_2(y) \leq y$.
\end{itemize}

Note that the function $y_s^s$ is strictly increasing whereas $y_b^s$ is strictly decreasing on its domain (see Figure~\ref{fig:x1x2y}).

\begin{figure}[ht]
  \begin{center}
  \includegraphics[width=0.5\textwidth]{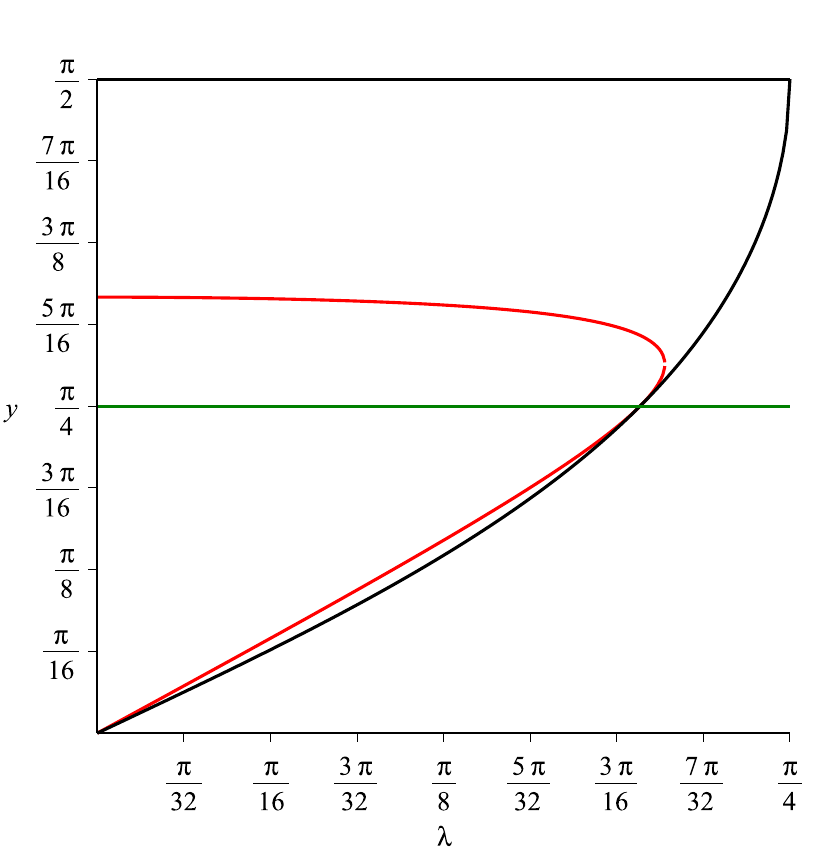} 
 \caption{The union of the graphs of the functions $y_s^s$ and $y_b^s$ decomposes the region consisting of the points with the possible values of $(\lambda, y)$ into three connected components, described in the previous list, which determine the relative position of $y$, $x_1^s(y)$ and $x_2^s(y)$. The region of the possible points is bounded by black curves. The red curve is the union of the graphs of $y_s^s$ and $y_b^s$. The green segment is the curve $y=\frac{\pi}{4}$.}
\label{fig:x1x2y}
\end{center}
\end{figure}

\begin{rem}\label{rem:regulartriangle}
Note that by the computations in Step 3, $T_1(y)^s$ or $T_2^s(y)$ is a regular triangle of edge length $2y$ if and only if $y = x_1^s(y)$ or $y=x_2^s(y)$. This yields that $T_{reg}^s(\rho)$ satisfies the conditions of the lemmas if and only if $y_s^s(\lambda) \leq \rho \leq y_b^s(\lambda)$.
\end{rem}

\emph{Step 4}: Computing the areas and the circumradii of $T_1^s(y)$ and $T_2^s(y)$.

For $i=1,2$, let $R_i^s(y)$ denote the the circumradius of $T_i^s(y)$.
Let the angle of $T_1^s(y)$ at $p$, $q_1$ and $q_2$ be denoted by $2\alpha$, $\beta$, $\beta$, respectively, and note that this implies that the angles of $T_2^s(y)$ at $p$, $-q_1$, $-q_2$ are $2\alpha$, $\pi-\beta$, $\pi-\beta$, respectively.
Then, by Girard's theorem, for $\area (T_1^s(y))$ we have
\[
\cos \frac{\area (T_1^s(y))}{2} = \sin (\gamma + \delta).
\]
Now, applying the spherical laws of sines and cosines, we obtain that $\sin \alpha = \frac{\sqrt{\sin^2y - \sin^2 \lambda}}{\sin y \cos \lambda}$, $\cos \alpha = \tan \lambda \cot y$, $\sin \beta = \frac{\tan \lambda}{\sin y}$ and$\cos \beta = \frac{\sqrt{\sin^2 y - \tan^2 \lambda}}{\sin y}$. Using trigonometric identities, from this we obtain
\[
\cos \frac{\area(T_1^s(y))}{2} = \frac{\sqrt{\sin^2 y - \sin^2 \lambda} \sqrt{\sin^2 y \cos^2 \lambda - \sin^2 \lambda} + \sin^2 \lambda \cos y}{\sin^2 y \cos^2 \lambda},
\]
and a similar computation yields
\[
\cos \frac{\area(T_2^s(y))}{2} = \frac{-\sqrt{\sin^2 y - \sin^2 \lambda} \sqrt{\sin^2 y \cos^2 \lambda - \sin^2 \lambda} + \sin^2 \lambda \cos y}{\sin^2 y \cos^2 \lambda}.
\]

Next, the circumradius $R$ of a spherical triangle with angles $\alpha, \beta, \gamma$ satisfies 
\[
\cot R = \sqrt{\frac{\sin(\alpha-P) \sin (\beta - P) \sin(\gamma-P)}{\sin P}},
\]
where $P= \frac{1}{2}(\alpha+\beta+\gamma-\pi)$ (see \cite{SV12}).
Thus, $R_1^s(y)$ satisfies
\[
\cot R_1^s(y) = \sqrt{\frac{\sin\left(\alpha+\beta - \frac{\pi}{2}\right) \sin^2 \left(\frac{\pi}{2}-\alpha \right)}{\sin\left(\alpha+\beta-\frac{\pi}{2} \right)}}=\frac{\cos \alpha \cos(\beta-\alpha)}{\sqrt{-\cos(\beta-\alpha) \cos(\beta+\alpha)}}.
\]
By our formulas for $\sin\alpha, \cos \alpha, \sin \beta, \cos \beta$ and using trigonometric identities, we obtain that
\[
-\cos(\beta-\alpha) \cos(\beta+\alpha) = \tan^2 \lambda.
\]
Substituting back and using the addition formula for the cosine of the difference of two angles, we obtain
\[
\cot R_1^s(y) = \frac{\sin \lambda}{\cos^2 \lambda} \left( \frac{\cos^2 y}{\sin^3 y} \sqrt{\sin^2 y \cos^2 \lambda - \sin^2 \lambda} + \frac{\cos y}{\sin^3 y} \sqrt{\sin^2 y - \sin^2 \lambda} \right),
\]
and a very similar computation yields
\[
\cot R_2^s(y) = \frac{\sin \lambda}{\cos^2 \lambda} \left( -\frac{\cos^2 y}{\sin^3 y} \sqrt{\sin^2 y \cos^2 \lambda - \sin^2 \lambda} + \frac{\cos y}{\sin^3 y} \sqrt{\sin^2 y - \sin^2 \lambda} \right).
\]

\emph{Step 5}: Investigating the monotonicity properties of the areas and the circumradii of $T_1^s(y)$ and $T_2^s(y)$.

Here we present the computation for $T_1^s(y)$.

Denoting $\sin \lambda$ and $\sin y$ by $L$ and $Y$, respectively, we obtain that the value of $\cos \frac{\area (T_1^s(y))}{2}$ can be written as
\[
f(Y)= \frac{\sqrt{Y^2-L^2}\sqrt{Y^2(1-L^2)-L^2} + L^2 \sqrt{1-Y^2}}{Y^2 \sqrt{1-L^2}}.
\]
Then we have
\[
f'(Y) = \frac{L^2 \left( A-B\right)}{\sqrt{Y^2-L^2} \sqrt{Y^2(1-L^2)-L^2} \sqrt{1-Y^2} (1-L^2)},
\]
where $A=(2Y^2-2L^2-Y^2L^2)\sqrt{1-Y^2}$ and $B=(2-Y^2)\sqrt{Y^2-L^2}\sqrt{Y^2(1-L^2)-L^2}$ are positive for any $0 < L < \frac{1}{\sqrt{2}}$ and $ \frac{L}{\sqrt{1-L^2}} < Y < 1$. On the other hand,
\[
A^2-B^2 = -Y^6 (1-L^2)(Y^2-2L^2),
\]
implying that $\area (T_1^s(y))$ is strictly decreasing on the interval $y \in [\arctan \sin \lambda, \arcsin (\sqrt{2} \sin \lambda)]$, and it is strictly increasing on $y \in \left[\arcsin (\sqrt{2} \sin \lambda), \frac{\pi}{2}\right]$.
Using a similar computation we have that $\area (T_2^s(y))$ is strictly increasing on its whole domain.

Next, we present the computation for $R_1^s(y)$. Using the notation $L = \sin \lambda$ and $Y = \sin y$, we have that $\cot R_1^s(y)$ is equal to
\[
g(Y)= \frac{1-Y^2}{Y^3}\sqrt{Y^2(1-L^2)-L^2}+\frac{\sqrt{1-Y^2}}{Y^3}\sqrt{Y^2-L^2}.
\]

We find the maximum of this expression for $\frac{L}{\sqrt{1-L^2}} < Y < 1$, where $L > 0$ is fixed.
To do it, we intend to examine first the condition when the squares of the derivatives of the first and the second members of $g(Y)$ are equal.
With the notation $Z=Y^2-\frac{5L^2}{3}$, this leads to the cubic equation $Z^3 + \left( 2L^2 - \frac{10L^4}{3} \right)Z + \frac{L^4}{3}  -\frac{25}{27} L^6=0$.
The discriminant of this equation is negative on the interval $L \in \left(0, \frac{1}{\sqrt{2}} \right)$, implying that it has one real root.
This root yields the unique positive solution $\sin y_{min}^s(\lambda))$ for $Y$. Numeric computations show that  
for any $0 < \lambda \leq \frac{\pi}{4}$,
\[
\arcsin \tan \lambda \leq y_{min}^s(\lambda) \leq \arcsin (\sqrt{2} \sin \lambda),
\]
and for any $0 < \lambda \leq \arcsin \frac{3}{5}$, we have $x_1^s(y_{min}^s(\lambda)) < y_s^s(\lambda) \leq y_{min}^s(\lambda)$, with equality on the right if and only if $\lambda = \arcsin \frac{1}{\sqrt{3}}$. In particular, checking the sign of $g'(Y)$, we obtain that
$R_1^s(y)$ is strictly decreasing on the interval $[\arcsin \tan \lambda, y_{min}^s(\lambda)]$ and strictly increasing on $\left[y_{min}^s(\lambda), \pi/2 \right)$.

A similar computation shows that $R_2^s(y)$ is strictly increasing on the interval $\left[ \arcsin \tan \lambda , \frac{\pi}{2} \right)$.

\emph{Step 6}: A case analysis to prove the lemmas.


We describe the analysis to find the minimum area triangles from amongst $T_1^s(y)$ and $T_2^s(Y)$ satisfying the conditions of Lemma~\ref{lem:densitySph}. To finish the proof of Lemma~\ref{lem:tightnessSph}, we can use a similar argument.

\textbf{Case 1}: $0 < \rho \leq \frac{\pi}{4}$.
Then $x_2(y) \geq \rho$ is satisfied for all values of $y$, and thus, $T_2^s(y)$ satisfies the conditions if and only if $y \geq \max \{ \rho, \arcsin \tan \lambda \}$.

\emph{Subcase 1.1}: $\lambda \leq \rho \leq y_s^s(\lambda)$.
By Step 2, $T_1^s(y)$ contains its circumcenter of and only if $y \leq \arcsin (\sqrt{2} \sin \lambda)$, and $T_2^s(y)$ contains it for all values of $y$.
Furthermore, the sides of $T_1^s(y)$ are at least $2\rho$ if and only if $y \geq \rho$, and $x_1(y) \geq \rho$.
By Step 3, these inequalities are equivalent to $y \geq \max \{ \rho, \arcsin \tan \lambda \}$, and $y \leq x_1^s|_{S_1}^{-1}(\rho)$ or $y \geq x_1^s|_{S_2}^{-1}(\rho)$.
Since $y \leq \arcsin (\sqrt{2} \sin \lambda)$, $y \geq x_1^s|_{S_2}^{-1}(\rho)$ is not satisfied for any value of $y$.
Thus, we need to find the minimum area triangle from amongst the $T_1^s(y)$ with $y \geq \max \{ \rho, \arcsin \tan \lambda \}$ and $y \leq x_1^s|_{S_1}^{-1}(\rho)$, and the $T_2^s(y)$ with $y \geq \max \{ \rho, \arcsin \tan \lambda \}$.
Observe that 
by the monotonicity properties of the function $x_1^s|_{S_1}^{-1}$ and the fact that $y_s^s(\lambda) < \arcsin (\sqrt{2} \sin \lambda)$, we have that
\[
x_1^s|_{S_1}^{-1}(\rho) \geq x_1^s|_{S_1}^{-1}(y_s^s(\lambda)) = y_s^s(\lambda) \geq \max \{ \arcsin \tan \lambda, \rho \}. 
\]
Thus, by Step 5, among the above triangles $T_1^s(x_1^s|_{S_1}^{-1}(\rho))$ has minimal area.

\emph{Subcase 1.2}: $y_s^s(\lambda) < \rho \leq \frac{\pi}{4}$.
Here $T_{reg}^s(\rho)$ satisfies the conditions, and thus, this triangle has minimal area.

\textbf{Case 2}: $\frac{\pi}{4} < \rho < \frac{\pi}{2}-\lambda$.
Note that in this case no triangle $T_1^s(y)$ satisfies the conditions. We also note that $y_b^s(\lambda) < \frac{\pi}{2}-\lambda$ for any value of $\lambda$.

\emph{Subcase 2.1}: $\lambda \leq \arcsin \frac{3}{5}$ and $y_s^s(\lambda) \leq \rho \leq y_b^s(\lambda)$. (Here we note that if $\lambda \leq \arcsin \frac{1}{\sqrt{3}}$, then $y_s^s(\lambda) \leq \rho$ holds for all $\rho \geq \frac{\pi}{4}$.)

Like in Subcase 1.2, $T_{reg}^s(\rho)$ satisfies the conditions, and thus, this triangle has minimal area.

\emph{Subcase 2.2}: if $\lambda > \arcsin \frac{3}{5}$, or $\arcsin \frac{1}{\sqrt{3}} \leq \lambda \leq \arcsin \frac{3}{5}$ and  $\rho < y_s(\lambda)$, or $0 \leq \lambda < \arcsin \frac{3}{5}$ and $\rho > y_b(\lambda)$. Then $T_{reg}^s(\rho)$ does not satisfy the conditions.

The triangle $T_2^s(y)$ satisfies the conditions if and only if $y \geq \max \{ \arcsin \tan \lambda, \rho\}$ and $x_2^s|_{S_1}^{-1}(\rho) \leq y \leq x_2^s|_{S_2}^{-1}(\rho)$. By our conditions, $x_2^s(\rho) < \rho$, implying that $\rho < x_2^s|_{S_1}^{-1}(\rho)$ or $\rho > x_2^s|_{S_2}^{-1}(\rho)$ if $\rho \leq \arcsin (\sqrt{2} \sin \lambda)$ or $\rho \geq \arcsin (\sqrt{2} \sin \lambda)$, respectively. Since $\rho \leq \frac{\Pi}{2}-\lambda$, and $y_b(\lambda)= \arcsin (\sqrt{2} \sin \lambda) = \arcsin (\sqrt{2} \sin \lambda)$ if $\lambda=\arcsin \frac{1}{\sqrt{3}}$, we have that $\rho > x_2^s|_{S_2}^{-1}(\rho)$ if and only if $0 \leq \lambda < \arcsin \frac{1}{\sqrt{3}}$ and $\rho > y_b^s(\lambda)$, and $\rho < x_2^s|_{S_2}^{-1}(\rho)$ otherwise (see Figure~\ref{fig:bigrhoregions}). In both cases, by Step 5, the solution is $T_2^s(\rho)$.

\begin{figure}[ht]
  \begin{center}
  \includegraphics[width=0.5\textwidth]{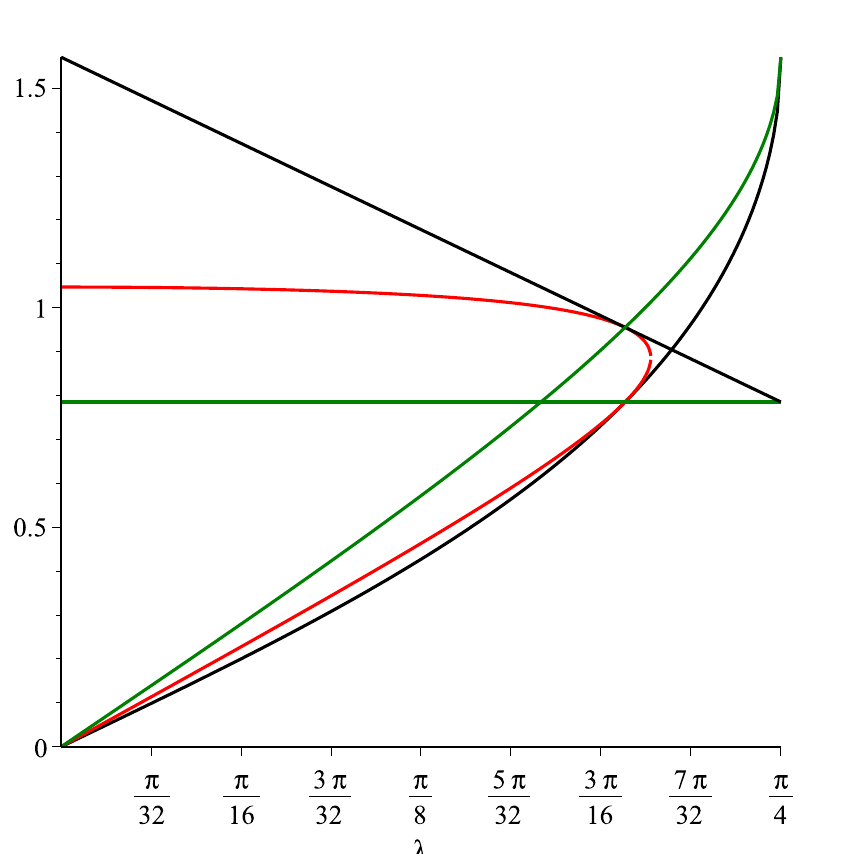} 
 \caption{An illustration for Subcase 2.2. The boundary of the region of the possible parameter values is drawn with black color. The red curve is the union of the graphs of $y_s^s$ and $y_b^s$. The green curves are the segment $y=\frac{\pi}{4}$, and the curve $y=\arcsin(\sqrt{2}\sin (\lambda))$. Observe that the last curve passes through the intersection of the segment $y=\frac{\pi}{2}-\lambda$ and the curve $y=y_b(\lambda)$.}
\label{fig:bigrhoregions}
\end{center}
\end{figure}

\textbf{Case 3}: $\rho > \frac{\pi}{2}-\lambda$. 
In this case $x_2^s(y) < \rho$ for all values of $y$, implying that there is no triangle satisfying the conditions.
\end{proof}

\section{Proofs of Theorems 1,2,6,7}\label{Euclidean etc}

The proofs in question are straightforward modifications of the proofs of Theorems 4 and 5, and are based on the use of the refined Moln\'ar decomposition.
In particular, let $\M \in \{ \Eu^2 , \HH^2 \}$. Like in Section~\ref{sec:proof45}, we assume that $\F$ is a $(2 R_{\rho})$-locally $\lambda$-separable packing of spherical caps of radius $\rho$ in $\M$, where $R_{\rho}= \arcsinh (\sqrt{2} \sinh \rho)$ if $\M = \HH^2$, and $R_{\rho} = \sqrt{2} \rho$ if $\M = \Eu^2$. In addition, we assume that $\F$ is \emph{$(2R_{\rho})$-saturated} in  $\M$, i.e., every point $p$ is at distance at most $2R_{\rho}$ from the center of a cap in $\F$. Thus, the radius of any circumdisk of the cells in the Delaunay decomposition of the set $X$ of the centers of the elements of $\F$ is at most $2 R_{\rho}$. 
Then the refined Moln\'ar decomposition of $\M$ defined by $X$ consists of two types of cells $P$:
\begin{itemize}
\item[(i)] $P$ has circumradius at most $R_{\rho}$, in this case we say that $P$ is type 1, or
\item[(ii)] it is of the form $P=\cl \left( \conv \{ v,c_i,c_j\} \setminus \conv \{ v',c_i,c_j\} \right)$, where $c_i,c_j \in X$, $v$ is the circumcenter of a Delaunay cell with $c_i$ and $c_j$ as vertices and with circumradius at least $R_{\rho}$, and $d(v',c_i) = d(v',c_j)$. In this case we say that $P$ is type 2.
\end{itemize}

Like in the spherical case, any type 1 cell is either a triangle containing its circumcenter, or it can be decomposed into two triangles containing their circumcenters. Thus, we assume that any type 1 cell is a triangle containing its circumcenter.

To prove the theorems for type 2 cells, we need the following lemma, which is a variant of \cite[Lemma 3]{BL23}.

\begin{lem}\label{lem:BL23nonSph}
Let $Q$ be an isosceles triangle in $\M$, with vertices $q_1, q_2, p$, where $p$ is the apex. Let the length of the legs of $Q$ be $x$, and that of the base be $y$, where $\rho \leq \frac{y}{2}< x$ . For $i=1,2$, let $S_i$ denote the disk of radius $\rho$, centered at $q_i$. Let
\[
f(x,y) =
\frac{\area(S_1 \cup S_2)}{\area(Q)}
\]
denote the density of the packing $\{ S_1, S_2 \}$ in $Q$ as a function of $x$ and $y$. Then $f(x,y)$ is a strictly decreasing function of both $x$ and $y$.
\end{lem}

This lemma is straightforward to prove if $\M = \Eu^2$, using an elementary computation. If $\M = \HH^2$, then a slight modification of the proof of Lemma 3 of \cite{BL23} can be applied, which we leave to the reader.
Thus, it is sufficient to prove the assertion for type 1 triangles.

First, we prove Theorem 1 for type 1 triangles. Assume that $\M = \Eu^2$ and $\rho=1$. Let $T$ be any type 1 triangle. Since the sides of $T$ has lengths at least $2$, it follows that $\area(T) \geq \frac{\sqrt{3}}{4}$. Similarly, since the Euclidean disks centered at the vertices of $T$ are a $\lambda$-separable system, we have that at least two heights of $T$ are at least $2\lambda$, implying that $\area(T) \geq \frac{\lambda}{2}$. Thus, $\area(T) \geq \max \left\{ \frac{\sqrt{3}}{4}, \frac{\lambda}{2} \right\}$. On the other hand, if $0 \leq \lambda \leq \frac{\sqrt{3}}{2}$, then the packing of unit disks whose Delaunay cells are regular triangles is a $\lambda$-separable packing, and if $\frac{\sqrt{3}}{2} < \lambda \leq 1$, then the same is true for the packing of unit disks whose Delaunay cells are isosceles 
triangles whose legs are of length $2$ and the length of the corresponding heights is $2 \lambda$, i.e. if they are congruent to $T^e(\sqrt{2-2\sqrt{1-\lambda^2}})$. This proves Theorem 1. To prove Theorems 2, 6 and 7, it is sufficient to prove the following lemmas.

\begin{lem}\label{lem:tightnessEu}
Let $T$ be a non-obtuse triangle in $\Eu^2$ with edge lengths at least $2$, and having two heights of lengths at least $2\lambda$. Let $R(T)$ denote the circumradius of $T$.
Then
\begin{equation}\label{eq:tightness2}
R(T) \geq
\left\{
\begin{array}{l}
\frac{2}{\sqrt{3}}, \hbox{ if } 0 \leq \lambda \leq \frac{\sqrt{3}}{2},\\
\frac{\sqrt{2-2\sqrt{1-\lambda^2}}}{\lambda}, \hbox{ if } \frac{\sqrt{3}}{2} \leq \lambda \leq \frac{2\sqrt{2}}{3},\\
\frac{3\sqrt{3} \lambda}{4}, \hbox{ if } \frac{2\sqrt{2}}{3} \leq \lambda \leq 1.
\end{array}
\right.
\end{equation}
Furthermore, we have equality in one of the three cases if and only if $T$ is congruent to $T^2_{reg}(1)$, or $T^e(\sqrt(2-2\sqrt{1-\lambda^2}))$, or $T^e\left( \sqrt{\frac{3}{2}} \lambda \right)$, respectively.
 \end{lem}

\begin{lem}\label{lem:densityH}
Let $T = \conv \{a,b,c\} \subset \HH^2$ be a hyperbolic triangle with edge lengths at least $2\rho$ that contains its circumcenter. For $x \in \{ a,b,c\}$, let $C_x$ denote the closed hyperbolic disk of radius $\lambda \leq \rho$ centered at $x$. Assume that there are lines $L_a$ and $L_b$ such that $L_a$ separates $C_a$ from $C_b$ and $C_c$, and $L_b$ separates $C_b$ from $C_a$ and $C_c$. Let $\delta(T)$ denote the density of $\{ C_x, C_y, C_z \}$ with respect to $T$.
Then we have the following.
\begin{equation}\label{eq:densitylemH}
\delta(T) \leq
\left\{
\begin{array}{l}
\delta(T^h(x^h|_{H_1}^{-1}(\rho))), \hbox{ if } \lambda \leq \rho \leq y_s^h(\lambda), \\
\delta(T_{reg}^s(\rho)), \hbox{ if } y_s^s(\lambda) < \rho.\\
\end{array}
\right.                                                                     
\end{equation}
Furthermore, in the above two cases we have equality if and only if $T$ is congruent to $T^h(x^h|_{H_1}^{-1}(\rho))$ or $T_{reg}^h(\rho)$, respectively.
\end{lem}

\begin{lem}\label{lem:tightnessH}
Let $T = \conv \{a,b,c\} \subset \HH^2$ be a hyperbolic triangle with edge lengths at least $2\rho$ that contains its circumcenter. For $x \in \{ a,b,c\}$, let $C_x$ denote the closed hyperbolic disk of radius $\lambda \leq \rho$ centered at $x$. Assume that there are lines $L_a$ and $L_b$ such that $L_a$ separates $C_a$ from $C_b$ and $C_c$, and $L_b$ separates $C_b$ from $C_a$ and $C_c$. Let $R(T)$ denote the circumradius of $T$.
Then we have the following.
\begin{equation}\label{eq:crlemH}
R(T) \geq
\left\{
\begin{array}{l}
R^h(y_{min}^h(\lambda)), \hbox{ if } \lambda \leq \rho \leq x^h(y_{min}^s(\lambda)),\\
R^h(x^h|_{H_1}^{-1}(\rho)), \hbox{ if } x^h(y_{min}^h(\lambda)) < \rho \leq y_s^h(\lambda),\\
R_{reg}^h(\rho), \hbox{ if } y_s^s(\lambda) < \rho.
\end{array}
\right.                                                                      
\end{equation}
Furthermore, in the above three cases we have equality if and only if $T$ is congruent to $T^h(y_{min}^h(\lambda))$, $T^h(x^h|_{H_1}^{-1}(\rho))$, or $T_{reg}^h(\rho)$, respectively.
\end{lem}

The proof of these lemmas follows the proof of Lemmas~\ref{lem:densitySph} and \ref{lem:tightnessSph}, and we only sketch them.

\begin{proof}[Proof of Lemmas~\ref{lem:tightnessEu}, \ref{lem:densityH} and \ref{lem:tightnessH}]
First, by the well known formula relating the angle defect of a hyperbolic triangle to its area, if $\M = \HH^2$, $\delta(T)$ is maximal if and only if $\area(T)$ is minimal among the triangles satisfying the conditions in Lemma~\ref{lem:densityH}.

Furthermore, similarly like in the previous section, we observe that both area and circumradius is minimized for some triangle satisfying the conditions in the lemmas.

Let $\M = \HH$, and let $T$ be a minimum area triangle satisfying the conditions. Let $H$ denote the open hyperbolic half plane containing $c$ in its interior and $a,b$ on its boundary. Recall the well known fact that for any $d \in H$ a hyperbolic triangle $T'=\conv \{ a,b d \}$ satisfies $\area(T')= \area(T)$ if and only if $d$ lies on the hypercycle through $c$ generated by the midpoints of $[a,c]$ and $[b,c]$. Let $G_0$ denote this hypercycle, and note that it is symmetric to the bisector $L$ of $[a,b]$. Applying the argument in the previous section, we obtain there is a hypercycle arc $G'$ in $G_0$ symmetric to $c_0$ such that for any $d \in C_0$, the triangle $\conv \{ a,b,d \}$ satisfies the conditions in Lemma~\ref{lem:densityH} if and only if $d \in G'$. The same argument can be applied for Lemmas~\ref{lem:tightnessEu} and \ref{lem:tightnessH} in which the circumcircle of $T$ plays the role of $G_0$.

From this, following the argument in the proof of Lemmas~\ref{lem:densitySph} and \ref{lem:tightnessSph}, we obtain that if $T$ has minimal circumradius or area, then $T$ is either
\begin{enumerate}
\item[(a')] a regular triangle of edge length $2\rho$, or
\item[(b')] $d(a,b)=d(b,c)$ and the line through the midpoints of $[a,c]$ and $[a,b]$ (respectively, the midpoints of $[b,c]$ and $[a,b]$) touches the three disks of radius $\lambda$ centered at the vertices of $T$.
\end{enumerate}

\emph{Step 1}: computing the edge lengths of the triangles satisfying (b').

Following the argument in Section~\ref{sec:proof45}, we obtain that for a hyperbolic triangle $T$,
\[
\sinh (2x) = \frac{\cosh \lambda \sinh^2 y}{\sqrt{\cosh^2 y - \cosh^2 \lambda}} = \frac{\cosh \lambda \sinh^2 y}{\sqrt{\sinh^2 y - \sinh^2 \lambda}},
\]
which is defined for any $y > \lambda$, yielding that the only hyperbolic triangle satisfying (b'), up to congruence, is $T^h(y)$.
In $\Eu^2$, we can similarly obtain that for any $y > \lambda$, up to congruence, there is a unique isosceles triangle satisfying (b'), namely $T^e(y)$.

\emph{Step 2}: Checking if $T^h(y)$ or $T^e(y)$ contains its circumcenter.

Applying the idea of the proof in Section~\ref{sec:proof45} yields that $T^h(y)$ contains its circumcenter if and only if $\arcsinh \tanh \lambda \leq y \leq \arcsinh (\sqrt{2} \sinh \lambda)$. Similarly, $T^e(y)$ contains its circumcenter if and only if $\lambda < y \leq \sqrt{2} \lambda$.

\emph{Step 3}: Finding the relations between the lengths of the sides of $T^h(y)$ and $T^e(y)$.

In $\HH^2$, a consideration like in the spherical case leads to the fact that the inequality $x^h(y) \geq y$ is satisfied 
if and only if
\[
y \leq y_s^h(\lambda)=\arcsinh \sqrt{ \frac{5 \sinh^2 \lambda - 3 + \sqrt{25 \sinh^4 \lambda + 34 \sinh^2 \lambda + 9}}{8} }.
\]
We note that like in the spherical case, we have that the hyperbolic disks of radius $\lambda$ centered at the vertices of $T_{reg}^h(\rho)$ are totally separable if and only if $ \rho \geq y_s^h(\lambda)$. In this case $T_{reg}^h(\rho)$ has minimal area and minimal circumradius among the triangles satisfying the conditions of the lemmas.

In the Euclidean plane, the inequality $x^e(y)  \geq y$ is satisfied if and only if $\lambda < y \leq \frac{2}{\sqrt{3}} \lambda$. Here, $T^e_{reg}(1)$ satisfies the conditions if and only if $\lambda \leq \frac{\sqrt{3}}{2}$. In this case $T^e_{reg}(1)$ has minimal circumradius.

\emph{Step 4}: Computing $\area(T^h(y))$, $R^y(y)$ and $R^e(y)$.

A slight modification of the computation in the spherical case shows that
\[
\cosh \frac{\area(T^h(y))}{2} = \frac{\sqrt{\sinh^2 y - \sinh^2 \lambda} \sqrt{\sinh^2 y \cosh^2 \lambda - \sinh^2 \lambda} + \sinh^2 \lambda \cosh y}{\sinh^2 y \cosh^2 \lambda}, \hbox{ and}
\]
\[
\coth R^h(y) = \frac{\sinh \lambda}{\cosh^2 \lambda} \left( \frac{\cosh^2 y}{\sinh^3 y} \sqrt{\sinh^2 y \cosh^2 \lambda - \sinh^2 \lambda} + \frac{\cosh y}{\sinh^3 y} \sqrt{\sinh^2 y - \sinh^2 \lambda} \right).
\]
Here we used the formula $\coth R= \sqrt{\frac{\sin(\alpha+P) \sin (\beta + P) \sin(\gamma+P)}{\sin P}}$ for the circumradius of a hyperbolic triangle with angles $\alpha, \beta, \gamma$, where $P=\frac{1}{2}(\pi-\alpha-\beta-\gamma)$ \cite{SV12}.
Similarly, for the circumradius $R^e(y)$ of $T^e(y)$, we obtain
\[
R^e(y) = \frac{y^3}{2 \lambda \sqrt{y^2-\lambda^2}}.
\]

\emph{Step 5}: Investigating the monotonicity properties of $\area(T^h(y))$, $R^h(y)$ and $R^e(y)$.

A simple computation following the proof in the spherical case shows that $\area(T^h(y))$ is strictly decreasing on 
$\left[ \arcsinh \tanh \lambda, \arcsinh ( \sqrt{2} \sinh \lambda) \right]$, and strictly increasing on $\left[ \arcsinh ( \sqrt{2} \sinh \lambda) ,\infty \right)$.

We sketch the computation for $R^h(y)$. Denoting $\sinh \lambda$ and $\sinh y$ by $L$ and $Y$, we obtain that $\coth R^h(y)$ can be written as
\[
h(Y)= \frac{(Y^2+1) \sqrt{Y^2 (L^2+1)-L^2}}{Y^3}+\frac{\sqrt{Y^2+1} \sqrt{Y^2-L^2}}{Y^3}.
\]
We examine first the condition when the squares of the derivatives of the first and the second members of $h(Y)$ are equal.
With the notation $Z=Y^2-\frac{5L^2}{3}$, this leads to the cubic equation $Z^3 - \left( 2L^2 + \frac{10L^4}{3} \right)Z -\frac{25}{27} L^6- \frac{L^4}{3}=0$.
The discriminant of this equation is positive, implying that it has three real roots. Solving it, we obtain that $Y$ has to satisfy one of the following:
\begin{eqnarray*}
Y^2 & = &\frac{5L^2}{3} + \frac{2}{3} \sqrt{10L^4+6L^2} \cos\left( \frac{1}{3} \arccos \left( \frac{(25 L^2+9) L}{4 \sqrt{2} (5L^2+3)^{3/2}} \right) \right),\\
Y^2 & = &\frac{5L^2}{3} + \frac{2}{3} \sqrt{10L^4+6L^2} \cos\left( \frac{1}{3} \arccos \left( \frac{(25 L^2+9) L}{4 \sqrt{2} (5L^2+3)^{3/2}} \right) -\frac{2\pi}{3} \right),\\
Y^2 & = &\frac{5L^2}{3} + \frac{2}{3} \sqrt{10L^4+6L^2} \cos\left( \frac{1}{3} \arccos \left( \frac{(25 L^2+9) L}{4 \sqrt{2} (5L^2+3)^{3/2}} \right) -\frac{4\pi}{3} \right)
\end{eqnarray*}
A numeric computation shows that for every $L > 0$ all three expressions exist, and the third expression is negative whereas the first two expressions are positive. On the other hand, an elementary computation shows that the derivatives of the first and the second member of $f(Y)$ cannot be zero simultaneously for $L > 0$, and recall that we examined only if the two derivatives are equal in absolute value. Since the above expressions are continuous functions of $L$ at which the two derivatives are equal in absolute value identically, it is sufficient to test whether the derivatives have the same or opposite signs for one arbitrary value $L > 0$. Testing it we obtain that $h'(Y)=0$ if and only if $Y^2$ is equal to the second expression. Thus, by $Y > 0$, we obtain that for any $\lambda > 0$, there is a unique stationary point of the function $\coth R^h(y)$, namely
\[
y_{min}^h(\lambda) = \arcsinh \sqrt{ \frac{5L^2}{3} + \frac{2}{3} \sqrt{10L^4+6L^2} \cos\left( \frac{1}{3} \arccos \left( \frac{(25 L^2+9) L}{4 \sqrt{2} (5L^2+3)^{3/2}} \right) -\frac{2\pi}{3} \right)  }  > \lambda ,
\]
where $L = \sinh \lambda$.
A simple computation yields that for any $L > 0$, $\lim_{Y \to L+0} f'(Y) = \infty$ and $\lim_{Y \to \infty} Y^2 \cdot f'(Y) = -1 < 0$, showing that $R^h(y)$ is strictly decreasing on the interval $(\lambda,y_{min}^h(\lambda)]$, and strictly increasing on the interval $[y_{min}^h(\lambda),\infty)$.

We note that numeric computations show that $\lambda < x^h(y_{min}^h(\lambda)) < y_s^h(\lambda) < y_{min}^h(\lambda) <\arcsinh(\sqrt{2} \lambda)$ for all $\lambda > 0$ (see Figure~\ref{fig:hyperbolic}).

\begin{figure}[ht]
  \begin{center}
  \includegraphics[width=0.5\textwidth]{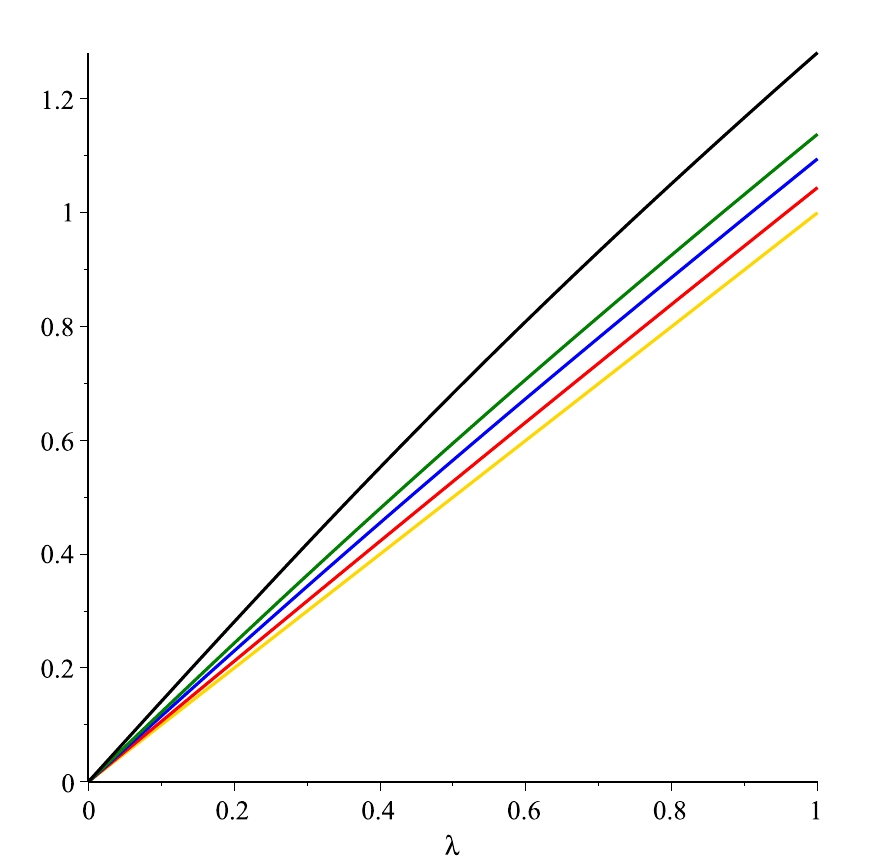} 
 \caption{An illustration for the curves $\lambda$ (yellow), $x^h(y_{min}^h(\lambda))$ (red), $y_s^h(\lambda)$ (blue), $y_{min}^h(\lambda)$ (green) and $\arcsinh(\sqrt{2}\sinh \lambda)$ (black) for a hyperbolic triangle.}
\label{fig:hyperbolic}
\end{center}
\end{figure}

In the Euclidean plane, examining the sign of the derivative of $R^e(y)$ shows that $R^e(y)$ is strictly decreasing on $\left( \lambda, \sqrt{\frac{2}{3}} \lambda \right]$, and strictly increasing on $\left[ \sqrt{\frac{2}{3}} \lambda, \infty \right)$.

\emph{Step 6}: A case analysis to prove the lemmas.

An investigation similar to the one in Section~\ref{sec:proof45} finishes the proof.
\end{proof}

\section{Proof of Theorem~\ref{contact-upper-bound}}

Theorems of Harborth \cite{Ha} and of Heitmann and Radin \cite{HeRa} imply in a straightforward way that $c_{\lambda}(n, B)=\lfloor 3n -\sqrt{12n-3}\rfloor$ holds for all $n>1$ and $0\leq \lambda\leq\frac{\sqrt{3}}{2}$.

Now, let $\frac{\sqrt{3}}{2}< \lambda\leq 1$. Then $\lfloor 2n - 2\sqrt{n}\rfloor\leq c_{\lambda}(n, B)$ follows in a straightforward way from Theorem 11 in \cite{Be21}. So, we are left to show that $c_{\lambda}(n, B)\leq 2n -\sqrt{\pi\lambda}\sqrt{n}+O(1)$ holds for $n>1$.

First, recall that the proof of Theorem~\ref{thm:densityEu} in Section~\ref{Euclidean etc} gives a proof of

\begin{lem}\label{Bezdek-Langi}
Let $\frac{\sqrt{3}}{2}<\lambda\leq 1$ and let $\mathcal{P}$ be a $\lambda$-separable packing of unit disks in $ \Eu^2$. Then the density of $\mathcal{P}$ in each cell of the refined Moln\'ar decomposition is at most $\frac{\pi}{4\lambda}$.
\end{lem}

Second, we use the proof technique of Theorem 4 of Eppstein \cite{Ep} as well as Lemma~\ref{Bezdek-Langi} for proving Lemma~\ref{eppstein-type}. In order to state it we need

\begin{defn}
Let $\mathcal{P}$ be a packing of unit disks in $ \Eu^2$. Then the plane graph $G_c(\mathcal{P})$ whose vertices are the center points of the unit disks of $\mathcal{P}$ and whose edges are the line segments connecting two vertices of $G_c(\mathcal{P})$ if and only if the corresponding two unit disks of $\mathcal{P}$ are tangent, is called the {\rm contact graph} of $\mathcal{P}$.
\end{defn}

\begin{lem}\label{eppstein-type}
Let $\frac{\sqrt{3}}{2}<\lambda\leq 1$ and let $\mathcal{P}$ be a $\lambda$-separable packing of $n>1$ unit disks in $ \Eu^2$. Then in the contact graph $G_c(\mathcal{P})$, the number of vertex-face incidences on the outer face of $G_c(\mathcal{P})$ is at least $2\sqrt{\pi\lambda}\sqrt{n}-O(1)$.
\end{lem}

\begin{proof} If the outer face of $G_c(\mathcal{P})$ has at least $2\sqrt{\pi\lambda}\sqrt{n}-O(1)$ vertex-face incidences, then we are done. So, assume that the outer face has less than $2\sqrt{\pi\lambda}\sqrt{n}$ vertex-face incidences. Then the perimeter of the outer face is less than $4\sqrt{\pi\lambda}\sqrt{n}$. Now, let $A$ be the complement of the outer face of $G_c(\mathcal{P})$ in $ \Eu^2$. Notice that if a unit disk of $\mathcal{P}$ has its center in the interior ${\rm int}(A)$ of $A$, then the open unit disk belongs to ${\rm int}(A)$. Furthermore, the cells of the refined Moln\'ar decomposition of the center points of the unit disks of $\mathcal{P}$ generate a decomposition of $A$. Next, Lemma~\ref{Bezdek-Langi} together with the property that the outer face has less than $2\sqrt{\pi\lambda}\sqrt{n}$ vertex-face incidences imply that

\begin{equation}\label{BL-inequality}
\area(A)\geq \frac{(n-2\sqrt{\pi\lambda}\sqrt{n})\pi}{\frac{\pi}{4\lambda}}=4\lambda(n-2\sqrt{\pi\lambda}\sqrt{n}).
\end{equation}

Finally, (\ref{BL-inequality}) and the isoperimetric inequality applied to $A$ yield for the perimeter ${\rm per}(A)$ of $A$ that
$$
{\rm per}(A)\geq \sqrt{4\pi \area(A)} \geq 2\sqrt{\pi}\sqrt{4\lambda}\sqrt{n-2\sqrt{\pi\lambda}\sqrt{n}}=4\sqrt{\pi\lambda}\sqrt{n-2\sqrt{\pi\lambda}\sqrt{n}}=4\sqrt{\pi\lambda}\sqrt{n}-O(1),
$$
finishing the proof of Lemma~\ref{eppstein-type}.

\end{proof}

Third, we need Lemma 5 of \cite{Ep} stated as follows.

\begin{lem}\label{edge-estimate}
Let $G$ be a triangle-free plane graph on $n$ vertices in which one face has $k$ vertex-face incidences. Then $G$ has at most $2n-\frac{k}{2}-2$ edges.
\end{lem}

Finally, notice that if $\frac{\sqrt{3}}{2}<\lambda\leq 1$ and $\mathcal{P}$ is a $\lambda$-separable packing of $n>1$ unit disks in $ \Eu^2$, then $G_c(\mathcal{P})$ is triangle-free. Thus, Lemma~\ref{eppstein-type} proves the existence of a large face in $G_c(\mathcal{P})$, and plugging the size of this face as the variable $k$ in Lemma~\ref{edge-estimate} shows that $c_{\lambda}(n, B)\leq 2n -\sqrt{\pi\lambda}\sqrt{n}+O(1)$ holds for $n>1$.
This completes the proof of Theorem~\ref{contact-upper-bound}.

\section{Appendix: The hyperbolic Moln\'ar decomposition}

This Appendix introduces a new decomposition technique in $\HH^2$, whose Euclidean analogue has been discovered by Moln\'ar \cite{Molnar}. It is obtained from the Delaunay decomposition.  For preciseness, we recall some elementary properties of a decomposition \cite{GSh}. First, we call a family $\mathcal{F}$ of countably many mutually nonoverlapping closed sets whose union covers $\HH^2$ a \emph{decomposition} or \emph{tiling} of $\S^2$. The elements of this family are called \emph{tiles} or \emph{cells}. In case of a \emph{convex decomposition} the tiles are assumed to be convex, implying that they are convex polygons. A decomposition is called \emph{locally finite} if any point has a neighborhood intersecting only finitely many tiles, and, if the tiles are polygons, it is called \emph{edge-to-edge}, if any edge of a tile coincides with exactly one edge of another tile. If $\mathcal{F}$ is edge-to-edge, the edges and the vertices of the tiles are called the \emph{edges} and the \emph{vertices} of $\mathcal{F}$, respectively.

Consider a discrete, countable point system $X = \{ p_1, \ldots, p_k, \ldots \}$ in $\HH^2$. It is well known that then there is a unique decomposition of $\conv(X)$ into convex polygonal tiles with the following properties:
\begin{enumerate}
\item the tiling is locally finite and edge-to-edge;
\item the vertices of every tile are points of $X$;
\item the circumdisk of every tile contains no point of $X$ in its interior, and only the vertices of the tile on its boundary. 
\end{enumerate}
This convex, locally finite and edge-to-edge tiling is called the \emph{Delaunay decomposition}, or \emph{$D$-decomposition} defined by $X$.
Now, by means of Lemma~\ref{lem:hyperbolicMolnar} we define another locally finite, edge-to-edge decomposition, which we call \emph{hyperbolic Moln\'ar decomposition} in short, $M$-decomposition.

Let $F$ be a cell of the $D$-decomposition, and let us denote the circumdisk of $F$ by $C_F \subset \HH^2$, and the center of $C_F$ by $o_F$.
If $F$ does not contain $o_F$, then there is a unique side of $F$ that separates it from $F$ in $C_F$. We call this side a \emph{separating side} of $F$. If $[p_i, p_j ]$ is a separating side of $F$, then we call the polygonal curve $[p_i, o ] \cup [o_F, p_j]$ the \emph{bridge} associated to the separating side $[p_i, p_j ]$ of $F$ (cf. Figure~\ref{fig:molnar}).

\begin{figure}[ht]
\begin{center}
 \includegraphics[width=0.2\textwidth]{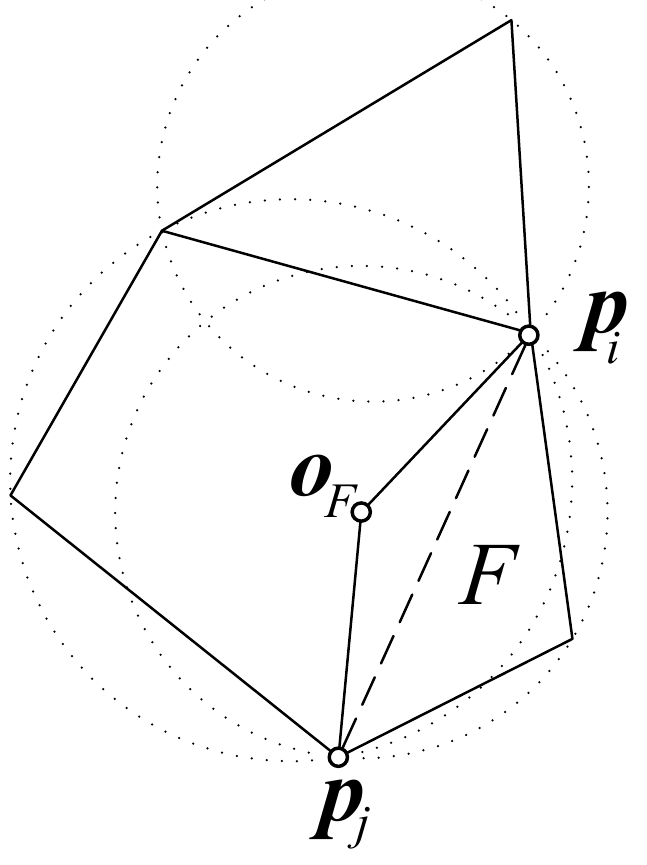}
 \caption{An illustration for the $M$-decomposition of a point system on $\HH^2$. In the figure, hyperbolic segments are represented with straight line segments; the edges of the cells, the circumcircles of the cells and the separating sides are denoted by solid, dotted and dashed lines, respectively.}
\label{fig:molnar}
\end{center}
\end{figure}

Our main lemma is the following.

\begin{lem}\label{lem:hyperbolicMolnar}
If we replace all separating sides of a $D$-decomposition by the corresponding bridges, we obtain a locally finite, edge-to-edge decomposition of $\HH^2$.
\end{lem}

\begin{proof}
As in \cite{Molnar}, the proof is based on showing the following two statements.
\begin{itemize}
\item[(a)] Bridges may intersect only at their endpoints.
\item[(b)] A bridge may intersect a side in the $M$-decomposition only at its endpoints.
\end{itemize}

To show (a), recall that for any cell $F$, the points of $X$ closest to $o_F$ are exactly the vertices of $F$. Thus, if $[o_F, p_i]$ and $[o_{F'}, p_j]$ are components of bridges with $p_i \neq p_j$ and $o_F \neq o_{F'}$, then the bisector of the segment $[ p_i, p_j]$ separates $[o_F, p_i]$ and $[o_{F'}, p_j]$, which, since $o_F \neq o_{F'}$, yields that the components are disjoint.

Now, we show (b), and let $[ p_i, p_j]$ be a separating side of the cell $F$. First, observe that by the definition of separating side, the triangle $T$ with vertices $p_i$, $p_j$ and $o_F$ intersects finitely many cells of the $D$-decomposition at a point different from $[ p_i, p_j]$, each of which is different from $F$. Let $F'$ be the cell adjacent to $F$ and having $[ p_i, p_j]$ as a side.
Note that both $o_F$ and $o_{F'}$ lie on the bisector of $[ p_i, p_j]$ such that if $m$ is the midpoint of this arc, then $o_F \in [o_{F'}, m ]$, implying that the radius of $C_{F'}$ is strictly greater than that of $C_F$. If, apart from $[ p_i, p_j]$, $F'$ contains $T$ in its interior, then (b) clearly holds for the bridge associated to $[p_i, p_j]$. On the other hand, if $F'$ does not contain $T$ in its interior apart from $[ p_i, p_j]$, then there is a separating side $[p_{i'},p_{j'}]$ of $F'$, removed during the construction, and this side is the unique side containing any point of the bridge apart from $p_i$ and $p_j$. Thus, no side of $F'$ in the $M$-decomposition contains an interior point of the bridge. Furthermore, observe that if any side in the $M$-decomposition contains an interior point of $[p_i,o_F] \cup [o_F, p_j]$, then it also contains an interior point of $[p_{i'}, o_{F'}] \cup [o_{F'}, p_{j'}]$. Consequently, we may repeat the argument with $F'$ playing the role of $F$, and since only finitely many cells may intersect $T$ at a point different from $[ p_i, p_j]$, and each is different from $F$, we conclude that $[ p_i, o_F] \cup [o_F, p_j]$ may intersect any side of the $M$-decomposition only at $p_i$ or $p_j$.
\end{proof}

\bigskip


\noindent K\'aroly Bezdek \\
\small{Department of Mathematics and Statistics, University of Calgary, Canada}\\
\small{Department of Mathematics, University of Pannonia, Veszpr\'em, Hungary\\
\small{E-mail: \texttt{kbezdek@ucalgary.ca}}

\bigskip

\noindent and

\bigskip

\noindent Zsolt L\'angi \\
\small{MTA-BME Morphodynamics Research Group and Department of Geometry}\\ 
\small{Budapest University of Technology and Economics, Budapest, Hungary}\\
\small{\texttt{zlangi@math.bme.hu}}

\end{document}